\documentclass[12pt]{article}
\usepackage[utf8]{inputenc}
\usepackage{authblk}
\usepackage{fullpage}

\usepackage{amsmath, amsfonts, amssymb, latexsym, amsthm}
\usepackage{algorithm}
\usepackage{algpseudocode}
\usepackage{esvect}
\usepackage{graphicx}
\usepackage{setspace}
\usepackage{url}
\usepackage{hyperref}
\usepackage{wrapfig}
\usepackage{tikz}
\usepackage[dvipsnames]{xcolor}
\usetikzlibrary{calc}
\usetikzlibrary{positioning}
\usepackage{tkz-graph}
\usetikzlibrary{shapes.geometric}
\usetikzlibrary{decorations.markings}
\usetikzlibrary{decorations.pathreplacing}
\usetikzlibrary{patterns}
\usepackage[table]{xcolor}
\usepackage{mathtools}
\theoremstyle{definition}
\newtheorem{theorem}{Theorem}[section]
\newtheorem{lemma}[theorem]{Lemma}

\newtheorem{corollary}[theorem]{Corollary}
\newtheorem{conjecture}[theorem]{Conjecture}

\newtheorem{question}[theorem]{Question}

\newcommand{\deviations}{\Upsilon}

\usepackage{comment}

\title{Locally Semi-Equitable Colourings of BIBDs}

\author[1]{Andrea C. Burgess}
\author[2]{William Kellough}
\author[2]{David A. Pike}

\affil[1]{Department of Mathematics and Statistics, University of New Brunswick, Saint John, NB, E2L 4L5, Canada}
\affil[2]{Department of Mathematics and Statistics, Memorial University of Newfoundland, St. John’s, NL, A1C 5S7, Canada}

\begin{document}

\maketitle

\begin{abstract}
    We study $\ell$-colourings of $(v,k,\lambda)$-BIBDs (balanced incomplete block designs) where, within each block, one colour is absent and each of the $\ell-1$ other colours appears exactly $\frac{k}{\ell-1}$ times. We establish several necessary conditions for such colourings to exist. We also use these coloured BIBDs to provide new necessary conditions for the existence of Hadamard matrices, affine planes, and twin prime powers. 
\end{abstract}

\begin{quote} \small {\bf Keywords:} colouring; balanced incomplete block design; Hadamard matrices; affine planes; twin prime powers

\smallskip 

\end{quote}

\section{Introduction}\label{sec: intro}

If $v,k,\lambda \in \mathbb{Z}^+$ with $2\leq k< v$, then a \emph{$(v,k,\lambda)$-BIBD (balanced incomplete block design)} is a pair $(V, \mathcal{B})$ where $V$ is a set of $v$ points and $\mathcal{B}$ is a collection of $k$-element subsets of $V$, called \emph{blocks}, such that for any pair of distinct $x,y\in V$, there are exactly $\lambda$ blocks containing both $x$ and $y$. The value $\lambda$ is called the \emph{index} of the design. If $C = \{c_1, \dots, c_{\ell}\}$, then an \emph{$\ell$-colouring} of a BIBD $(V, \mathcal{B})$ is a surjective mapping $f: V\to C$, where the elements of $C$ are called \emph{colours}. The \emph{colour class} of $c_i \in C$ is the set of all points in $V$ coloured $c_i$. A \emph{locally equitable colouring} of $(V,\mathcal{B})$ is a colouring such that for any $B\in \mathcal{B}$ and any two colour classes $C_1$ and $C_2$, $\big| |C_1 \cap B| - |C_2 \cap B|\big| \leq 1$.

The study of locally equitably coloured designs initially started with cycle systems, which were considered in \cite{ABLW04, ABW07, BM21, BM24, LW05, W05, W06}. Locally equitable colourings of group divisible designs and maximum packing designs are studied in \cite{BDDKLPY26}. Luther and Pike \cite{L16, LP16} investigated such colourings on BIBDs and characterized all BIBDs which admit a locally equitable colouring. We briefly note that we are using different terminology than in \cite{L16, LP16}. For what they called ``equitable'' colourings, we use the term ``locally equitable'' because ``equitable'' colourings in the hypergraph colouring literature may instead refer to colourings such that for any two colour classes $C_1$ and $C_2$, $\big||C_1| - |C_2|\big| \leq 1$. For a survey on these kinds of colourings for $(v,3,\lambda)$-BIBDs, see Sections 18.2--18.4 in \cite{CR99}. In the literature on colouring block designs, an ``equitable'' colouring may also refer to a specific kind of colouring where the blocks are assigned colours instead of the points (see for example \cite{GHMR08}), but such colourings are not studied in this paper.

Luther and Pike \cite{L16, LP16} showed that a BIBD admits a locally equitable $\ell$-colouring if and only if $\ell = v$ or it is a $(k+1, k, u(k-1))$-BIBD such that $k \equiv \ell - 1 \pmod{\ell}$ and $u \in \mathbb{Z}^+$. Because of how restrictive admitting a locally equitable colouring is for the parameters of a BIBD, one could ask: Which BIBDs can be coloured in a way that is close to a locally equitable colouring?

In this paper, we generalize locally equitable colourings to \emph{locally semi-equitable colourings}. Let $(V,\mathcal{B})$ be a $(v,k,\lambda)$-BIBD with $\mathcal{B} = \{B_1. \dots, B_b\}$. For each $B_i \in \mathcal{B}$, let $n_i \in \mathbb{Z}_{\geq 0}$ and let $\deviations$ be the multiset $\{n_1, \dots, n_b\}$. An \emph{$\deviations$-locally semi-equitable colouring}, abbreviated \emph{$\deviations$-LSE}, of $(V, \mathcal{B})$ is a colouring such that for each $B_i \in \mathcal{B}$:
\begin{enumerate}
    \item[(1)] one colour is represented exactly $n_i$ times within $B_i$, and
    \item[(2)] for any two colours $c_{j_1}$ and $c_{j_2}$ that are not the colour from condition (1), their colour classes $C_{j_1}$ and $C_{j_2}$ satisfy $\big||C_{j_1} \cap B_i| - |C_{j_2} \cap B_i|\big| \leq 1$. 
\end{enumerate}
In the case that $n_i = t$ for all $1\leq i\leq b$ where $t\in \mathbb{Z}^+$ is a constant, we write $t$-LSE instead of $\deviations$-LSE. Notably, if for all $1\leq i\leq b$, $n_i \in \left\{ \left\lfloor\frac{k}{\ell}\right\rfloor, \left\lceil\frac{k}{\ell}\right\rceil \right\}$, then an $\deviations$-LSE $\ell$-colouring is a locally equitable colouring.

Let $\deviations = \{n_1, \dots, n_b\}$. We separate $\deviations$-LSE colourings into two types. An $\deviations$-LSE colouring of a $(v,k,\lambda)$-BIBD is \emph{uniform} if $\frac{k-n_i}{\ell-1} \in \mathbb{Z}$ for all $1\leq i\leq b$ and \emph{nonuniform} otherwise. We further abbreviate uniform and nonuniform $\deviations$-LSE to $\deviations$-ULSE and $\deviations$-NLSE respectively. In particular, in a $t$-ULSE $\ell$-colouring of a $(v,k,\lambda)$-BIBD, $(\ell-1)$ divides $(k-t)$. The main focus of this paper is on $0$-ULSE colourings.

In Section~\ref{sec: Preliminaries} we give preliminary results on BIBDs. We also obtain results that apply to all $\deviations$-LSE coloured BIBDs and all $t$-LSE coloured BIBDs. The rest of the paper focuses on $0$-ULSE colourings. In Section~\ref{sec: Necessary Conditions} we establish necessary conditions for $0$-ULSE $\ell$-coloured $(v,k,\lambda)$-BIBDs, including a relation between $v$, $k$, and $\ell$. In Section~\ref{sec: Construction} we present the main result of our paper, which is a general framework for constructing $0$-ULSE coloured BIBDs. In Section~\ref{sec: Applications of general construction} we show that Hadamard matrices and affine planes can be used with the general framework in Section~\ref{sec: Construction} to construct families of $0$-ULSE coloured BIBDs. The results of Section~\ref{sec: Applications of general construction} also establish new necessary conditions for the existence of Hadamard matrices and affine planes. In Section~\ref{sec: TPP Construction} we show how twin prime powers can be used to construct $0$-ULSE coloured BIBDs, which gives a new necessary condition for the existence of twin prime powers. We also show that there exist $0$-ULSE coloured BIBDs beyond those that can be obtained by the construction in Section~\ref{sec: Construction}. We finish the paper with open problems and future work in Section~\ref{sec: Open Problems}.

\section{Preliminaries}\label{sec: Preliminaries}

In this section we present some elementary properties of BIBDs and $\deviations$-LSE colourings. For more background on BIBDs, we direct the reader to \cite{S04}.

\begin{lemma}\label{lem: replication num of a BIBD}
    If $x$ is a point in a $(v,k,\lambda)$-BIBD then the number of blocks containing $x$ is exactly $$r = \frac{\lambda(v-1)}{k-1}.$$
\end{lemma}

We call $r$ the \emph{replication number} of a BIBD and we continue denoting the replication number with $r$ throughout this paper. 

\begin{lemma}\label{lem: num of blocks in a BIBD}
    If $b$ is the number of blocks in a $(v,k,\lambda)$-BIBD then $$b = \frac{vr}{k} = \frac{\lambda v(v-1)}{k(k-1)}.$$
\end{lemma}

We continue denoting the number of blocks in a design with $b$ throughout this paper. The following is a well known result by Fisher \cite{F40}.

\begin{lemma}[Fisher's Inequality]\label{lem: Fisher's ineq.}
    If a $(v,k,\lambda)$-BIBD exists with replication number $r$ and $b$ blocks, then $v\leq b$ and, consequently, $k \leq r$.
\end{lemma}

A $(v,k,\lambda)$-BIBD is called \emph{symmetric} if $v = b$, or equivalently $k=r$.

Next, we show that given an $\deviations$-LSE coloured BIBD, we can obtain other designs with particular colourings of interest. 

\begin{theorem}\label{thm: duplicating blocks}
    Let $\mathcal{D}$ be a $(v,k,\lambda)$-BIBD. If $\mathcal{D}$ admits an $\deviations$-LSE colouring, then for any $m\in \mathbb{Z}^+$ there exists a $(v,k,m\lambda)$-BIBD that admits an $\deviations$-LSE colouring. 
\end{theorem}

\begin{proof}
    Fix an $\deviations$-LSE colouring of $\mathcal{D}$. Let $\{B_1, \dots, B_b\}$ be the blocks of $\mathcal{D}$. Let $\mathcal{D}_m$ be the $(v,k,m\lambda)$-BIBD with blocks $\{B_{1_1}, B_{1_2}, \dots, B_{1_m}, B_{2_1}, \dots, B_{2_m}, \dots, B_{b_m}\}$ where $B_{i_j} = B_i$. If we assign the same colouring to $\mathcal{D}_m$ as we did to $\mathcal{D}$, then, since the blocks of $\mathcal{D}_m$ are copies of the blocks of $\mathcal{D}$, the colouring will be an $\deviations$-LSE colouring. 
\end{proof}

It is known that when given a BIBD, say $\mathcal{D}$, a different BIBD, $\mathcal{D}^\prime$, can be formed by letting the blocks of $\mathcal{D}^\prime$ be the complement of the blocks in $\mathcal{D}$. 

\begin{lemma}\label{lem: complement of a BIBD}
    Let $v,k,\lambda, b, r \in \mathbb{Z}^+$ with $k\leq v-2$. A $(v,k,\lambda)$-BIBD with $b$ blocks and replication number $r$ exists if and only if a $(v, v-k, b-2r+\lambda)$-BIBD exists.
\end{lemma}

\begin{theorem}\label{thm: complements of SE colourings}
    Let $v,k,\lambda, b, r, \ell \in \mathbb{Z}^+$ such that $k\leq v-2$ and $\ell$ divides $v$, let $\deviations = \{n_1, \dots, n_b\}$, and let $\deviations^\prime = \{\frac{v}{\ell} - n_1, \dots, \frac{v}{\ell}- n_b\}$. An $\deviations$-LSE $\ell$-colouring of a $(v,k,\lambda)$-BIBD with $b$ blocks, replication number $r$, and each colour class having the same size exists if and only if an $\deviations^\prime$-LSE $\ell$-colouring of a $(v, v-k, b-2r+\lambda)$-BIBD exists.
\end{theorem}

\begin{proof}
    Let $\mathcal{D}$ be a $(v,k,\lambda)$-BIBD with blocks $\{B_1, \dots, B_b\}$ and fix an $\deviations$-LSE $\ell$-colouring of $\mathcal{D}$ such that all colour classes have the same size. Consequently, every colour class has size $\frac{v}{\ell}$. By the definition of an $\deviations$-LSE colouring, for each $1\leq i\leq b$, one colour is represented $n_i$ times in the block $B_i$ while the rest of the colours are represented either $\left\lfloor \frac{k-n_i}{\ell -1}\right\rfloor$ times or $\left\lceil \frac{k-n_i}{\ell -1}\right\rceil$ times. 

    By Lemma~\ref{lem: complement of a BIBD}, we can construct a $(v, v-k, b-2r+\lambda)$-BIBD, $\mathcal{D}^\prime$, by taking the complement of the blocks in $\mathcal{D}$. For each $1\leq i\leq b$, let $B_i^\prime$ be the complement of the block $B_i$. Suppose we colour the points of $\mathcal{D}^\prime$ the same as in the $\deviations$-LSE colouring of $\mathcal{D}$. We claim that this yields an $\deviations^\prime$-LSE $\ell$-colouring of $\mathcal{D}^\prime$. To see this, consider any block $B^\prime_i$ of $\mathcal{D}^\prime$. Since one colour, say $c_j$, is represented $n_i$ times in $B_i$, $c_j$ is represented $\frac{v}{\ell}-n_i$ times in $B_i^\prime$. For any other colour, say $c_m$, if $c_m$ is represented $\left\lfloor \frac{k-n_i}{\ell -1}\right\rfloor$ times in $B_i$, then $c_m$ will be represented $\frac{v}{\ell} - \left\lfloor \frac{k-n_i}{\ell -1}\right\rfloor$ times in $B^\prime_i$. Similarly, if $c_m$ is represented $\left\lceil \frac{k-n_i}{\ell -1}\right\rceil$ times in $B_i$, then $c_m$ will be represented $\frac{v}{\ell} - \left\lceil \frac{k-n_i}{\ell -1}\right\rceil$ times in $B^\prime_i$. Thus for any two colours $c_{m_1}$ and $c_{m_2}$ with corresponding colour classes $C_{m_1}$ and $C_{m_2}$ where $j\notin \{m_1, m_2\}$, we have that either \[|C_{m_1} \cap B^\prime_i| - |C_{m_2} \cap B^\prime_i| = 0\] or \[ \big||C_{m_1} \cap B^\prime_i| - |C_{m_2} \cap B^\prime_i|\big| = \left(\frac{v}{\ell} - \left\lfloor \frac{k-n_i}{\ell -1}\right\rfloor\right) - \left(\frac{v}{\ell} - \left\lceil \frac{k-n_i}{\ell -1}\right\rceil \right) = 1.\] 

    Suppose instead we fix an $\deviations^\prime$-LSE $\ell$-colouring of $\mathcal{D}^\prime$. By taking the block complement of $\mathcal{D}^\prime$, we obtain $\mathcal{D}$. By the same argument as above, $\mathcal{D}^\prime$ being $\deviations^\prime$-LSE $\ell$-colourable implies that $\mathcal{D}$ admits a $\{\frac{v}{\ell} - (\frac{v}{\ell}-n_1), \dots, \frac{v}{\ell}-(\frac{v}{\ell}-n_b)\}$-LSE $\ell$-colouring. So $\mathcal{D}$ is $\deviations$-LSE $\ell$-colourable. 
\end{proof}

By considering the case where each $n_i \in \deviations$ does not depend on $i$, we obtain the following corollary from Theorem~\ref{thm: complements of SE colourings}. 

\begin{corollary}\label{cor: complements of t-LSE colourings}
    Let $v,k,\lambda, b, r, \ell, t\in \mathbb{Z}^+$ such that $k\leq v-2$ and $\ell$ divides $v$. A $t$-LSE $\ell$-colouring of a $(v,k,\lambda)$-BIBD with $b$ blocks, replication number $r$, and with each colour class having the same size exists if and only if a $(\frac{v}{\ell}-t)$-LSE $\ell$-colouring of a $(v, v-k, b-2r+\lambda)$-BIBD exists.
\end{corollary}

We will utilize Corollary~\ref{cor: complements of t-LSE colourings} over Theorem~\ref{thm: complements of SE colourings} since our focus is on $0$-LSE colourings.

\section{Necessary Conditions for \texorpdfstring{$0$}{0}-ULSE Coloured BIBDs}\label{sec: Necessary Conditions}

Since $k < v$ by the definition of a BIBD, every BIBD on $v$ points admits a $0$-LSE $v$-colouring. We call $0$-LSE $v$-colourings of a $(v,k,\lambda)$-BIBD \emph{trivial} and when $\ell < v$, we call $0$-LSE $\ell$-colourings \emph{nontrivial}. It follows immediately from the definition of $0$-LSE $\ell$-colourings that $\ell > 1$. We now show that $0$-LSE $2$-colourings do not exist. 

\begin{lemma}
    If there exists a $0$-LSE $\ell$-coloured BIBD, then $\ell > 2$.
\end{lemma}

\begin{proof}
    If a $0$-LSE $2$-coloured BIBD exists, then every block contains points that are all the same colour. Suppose for a contradiction, we have a $0$-LSE $2$-coloured BIBD. Recalling that a colouring is surjective, there must be two points, $x$ and $y$, that are coloured differently. By the definition of a BIBD, there exists a block containing both $x$ and $y$ and so we have a block containing both colours. This is a contradiction and so there are no $0$-LSE $2$-coloured BIBDs. 
\end{proof}

In \cite{L16, LP16}, it was shown that if $\ell \geq k$, then a locally equitable $\ell$-coloured $(v,k,\lambda)$-BIBD exists if and only if $\ell = v$. We use an analogous argument to give a similar result for $0$-LSE colourings. 

\begin{lemma}\label{lem: upp bound on num of colours nB = 0}
    Let $\mathcal{D}$ be a $(v,k,\lambda)$-BIBD and let $\ell \geq k+1$. A $0$-LSE $\ell$-colouring of $\mathcal{D}$ exists if and only if $\ell = v$.
\end{lemma}

\begin{proof}
    If $\ell = v$, then assigning every vertex its own unique colour will result in a $0$-LSE $\ell$-colouring of $\mathcal{D}$.

    Suppose instead that $k+1 \leq \ell < v$ and suppose for a contradiction there exists a $0$-LSE $\ell$-colouring of $\mathcal{D}$. Since there are fewer colours than points, there exist two points, say $x$ and $y$, such that $x$ and $y$ are assigned the same colour, say $c_1$. Let $B$ be a block containing the pair $xy$. By the definition of a $0$-LSE colouring, one colour, say $c_2$, is absent from $B$. Since $|B| = k$, this leaves at least $k-1$ colours to be assigned to $k-2$ elements in $B$ that are neither $x$ nor $y$. Thus, a second colour will be missing from $B$, say $c_3$. If $C_1$ and $C_3$ are the colour classes for the colours $c_1$ and $c_3$ respectively, then $|C_1 \cap B| - |C_3 \cap B| > 1$. This contradicts our assumption that the BIBD is $0$-LSE $\ell$-coloured. 
\end{proof}

We therefore have the following bound on the number of colours in any nontrivial $0$-LSE colouring of a BIBD. 

\begin{theorem}\label{thm: nontrivial LSE cols bounds on ell}
    If a nontrivial $0$-LSE $\ell$-coloured $(v,k,\lambda)$-BIBD exists, then $3\leq \ell \leq k$.
\end{theorem}

We now focus on $0$-ULSE colourings and show that for any given colour, the size of its colour class and the number of blocks containing the colour are constants that are independent of the given colour. To do this, we define two new parameters.

In any $0$-ULSE $\ell$-colouring of a BIBD, every block has $(\ell - 1)$ colours and $(\ell - 1)$ divides $k$. For a given block and a given colour $c_i$, $c_i$ either appears $\frac{k}{\ell-1}$ times or $0$ times. Let $\alpha_i$ (respectively $\gamma_i$) denote the number of blocks where $c_i$ appears $\frac{k}{\ell-1}$ (respectively $0$) times.

\begin{lemma}\label{lem: l-1 divides k, colour classes the same size, calculating alpha}
    In a $(v,k,\lambda)$-BIBD with a $0$-ULSE $\ell$-colouring, each colour appears in exactly \[\frac{r^2}{\lambda} - \frac{(\ell - 1) r^2}{\lambda k} + \frac{(\ell - 1)r}{k}\] blocks and each colour class has size $\frac{v}{\ell}$.
\end{lemma}

\begin{proof}

Let $q = \frac{k}{\ell - 1}$. Let $c_i$ be a colour with corresponding colour class $C_i$.

Since each colour appears either $q$ or $0$ times in a block,

\begin{equation}\label{eq: k=0 mod l-1, alpha + gamma = b}
    \alpha_i + \gamma_i = b.
\end{equation}

By counting, in two different ways, the number of pairs $(x,B)$ where $x$ is a point coloured $c_i$ and $B$ is a block containing $x$ we obtain the following.

\begin{equation}\label{eq: k=0 mod l-1, alpha q = rCi}
    \alpha_i q = r|C_i|
\end{equation}

By counting, in two different ways, the number of pairs $(xy, B)$ where $xy$ is an unordered pair of points both coloured $c_i$ and $B$ is a block containing the pair $xy$ we obtain the following.

\begin{equation}\label{eq: k=0 mod l-1, alpha qC2 = lambda |C_i|C2}
    \alpha_i \binom{q}{2} = \lambda \binom{|C_i|}{2}
\end{equation}

By combining Equations \eqref{eq: k=0 mod l-1, alpha q = rCi} and \eqref{eq: k=0 mod l-1, alpha qC2 = lambda |C_i|C2} we obtain \[r|C_i| \frac{q-1}{2} = \frac{\lambda}{2}|C_i| (|C_i| - 1).\] Solving for $|C_i|$ gives 
\begin{equation}\label{eq: formula for |Ci| in terms of krlambda}
    |C_i| = \frac{kr}{\lambda (\ell -1)} - \frac{r}{\lambda} + 1.
\end{equation}

Since the right-hand side is independent of $i$, every colour class has the same size. Therefore, $|C_i| = \frac{v}{\ell}$ for every colour class $C_i$. Substituting $|C_i| = \frac{kr}{\lambda (\ell -1)} - \frac{r}{\lambda} + 1$ into Equations \eqref{eq: k=0 mod l-1, alpha + gamma = b} and \eqref{eq: k=0 mod l-1, alpha q = rCi} gives the following values for $\alpha_i$ and $\gamma_i$.

\begin{align}
    \alpha_i &= \frac{r^2}{\lambda} - \frac{(\ell - 1) r^2}{\lambda k} + \frac{(\ell - 1)r}{k} \label{eq: alpha formula} \\
    \gamma_i &= b - \frac{r^2}{\lambda} + \frac{(\ell - 1) r^2}{\lambda k} - \frac{(\ell - 1)r}{k} \label{eq: gamma formula}
\end{align}

\end{proof}

Notably, $\alpha_i$ and $\gamma_i$ are also independent of $i$. So, instead of $\alpha_i$ and $\gamma_i$, we write $\alpha$ and $\gamma$ for the remainder of the paper. 

\begin{theorem}\label{thm: strong necessary condition for 0-LSE ell-col with ell-1 | k}
    If a $(v,k,\lambda)$-BIBD with a $0$-ULSE $\ell$-colouring exists, then 
    \begin{equation}\label{eq: k=0 mod l-1, soln for v}
        v = \frac{\ell(\ell - 2)k}{(\ell-1)^2- k}.
    \end{equation}
\end{theorem}

\begin{proof}
Since $\gamma$ measures the number of blocks any given colour does not appear in, and since exactly one colour is missing from each block, 
\begin{equation}\label{eq: k=0 mod l-1, sum gamma = b}
    \sum_{i=1}^\ell \gamma = b.
\end{equation}
Combining Equation~\eqref{eq: alpha formula}, Equation~\eqref{eq: k=0 mod l-1, alpha + gamma = b} after summing both sides from $1$ to $\ell$, and Equation~\eqref{eq: k=0 mod l-1, sum gamma = b} yields 
\[b(\ell - 1) = \sum_{i=1}^\ell \alpha = \frac{\ell r^2}{\lambda} - \frac{\ell(\ell - 1) r^2}{\lambda k} + \frac{\ell(\ell - 1)r}{k}.\]
By using Lemmas~\ref{lem: replication num of a BIBD} and \ref{lem: num of blocks in a BIBD} and rearranging to solve for $v$, we get
\begin{equation*}
    v = \frac{\ell(\ell - 2)k}{(\ell - 1)^2 - k}. \qedhere
\end{equation*}
\end{proof}

We briefly note that Equation~\eqref{eq: k=0 mod l-1, soln for v} can be obtained in a different way by rearranging Equation~\eqref{eq: formula for |Ci| in terms of krlambda} after substituting $|C_i| = \frac{v}{\ell}$. We leave this to the reader to verify.

If we substitute $v = \ell$ into Equation~\eqref{eq: k=0 mod l-1, soln for v}, solving for $v$ leads to $v = k+1$. The only BIBDs that have $k+1$ points and block size $k$ are BIBDs of the form $(k+1, k, u(k-1))$ where $u \in \mathbb{Z}^+$. Furthermore, assigning a unique colour to every point of a $(k+1, k, u(k-1))$-BIBD produces a $0$-ULSE colouring. Therefore, a BIBD admits a trivial $0$-ULSE colouring if and only if it is a $(k+1, k, u(k-1))$-BIBD where $u\in \mathbb{Z}^+$. This is the same family of BIBDs, as mentioned in Section~\ref{sec: intro}, that admit a locally equitable $\ell$-colouring where $\ell < v$. This is not a coincidence. Every block of a $0$-ULSE $v$-coloured BIBD will have one colour absent and the rest appearing exactly once. Therefore, for any block $B$ and colour classes $C_i$ and $C_j$, $||C_i \cap B| - |C_j \cap B|| \leq 1$. So every trivial $0$-ULSE colouring is a locally equitable colouring.

Notably, not every BIBD that satisfies Equation~\eqref{eq: k=0 mod l-1, soln for v} admits a $0$-ULSE colouring. As an example, consider the BIBDs with parameters $(15,8,4)$. The complement of a $(15,8,4)$-BIBD is a $(15,7,3)$-BIBD, and so there are the same number of $(15,8,4)$-BIBDs up to isomorphism as there are $(15,7,3)$-BIBDs. In Table 1.30 of Section II.1.2 in \cite{HCD}, each of the five nonisomorphic $(15,7,3)$-BIBDs are listed. Through a computational search, we have determined that the first $(15,7,3)$-BIBD listed has a $3$-ULSE $5$-colouring with all colour classes of size three while the rest do not. Therefore, by Corollary~\ref{cor: complements of t-LSE colourings}, there is exactly one $(15,8,4)$-BIBD with a $0$-ULSE $5$-colouring while every other $(15,8,4)$-BIBD does not have such a colouring.

Using Theorem~\ref{thm: strong necessary condition for 0-LSE ell-col with ell-1 | k}, we can obtain bounds on the block size of any design that admits a $0$-ULSE colouring.

\begin{lemma}\label{lem: k=0 mod l-1, bounds on k}
    If a $(v,k,\lambda)$-BIBD admits a nontrivial $0$-ULSE $\ell$-colouring, then \[2(\ell - 1) \leq k \leq (\ell - 1)(\ell - 2).\]
\end{lemma}

\begin{proof}
    Since the colouring is nontrivial, we have $k \geq \ell$ by Lemma~\ref{lem: upp bound on num of colours nB = 0}. Necessarily, $k\equiv 0 \pmod{(\ell - 1)}$ and so the smallest possible value of $k$ is $2(\ell -1)$. 

    It must hold that the denominator of Equation~\eqref{eq: k=0 mod l-1, soln for v} is positive since its numerator is always positive for all $\ell \geq 3$. Thus 
    \[k \leq (\ell - 1)^2 - 1.\]

    Again, since $k\equiv 0 \pmod{(\ell - 1)}$, the largest possible value for $k$ is $(\ell - 1)^2 - (\ell - 1) = (\ell - 1)(\ell - 2)$. 
\end{proof}

As a consequence of Lemma~\ref{lem: k=0 mod l-1, bounds on k}, we can give tighter bounds on $\ell$ than in Theorem~\ref{thm: nontrivial LSE cols bounds on ell} when only considering $0$-ULSE colourings.

\begin{theorem}\label{thm: 0-ULSE bounds on ell}
    If there exists a nontrivial $0$-ULSE $\ell$-coloured $(v,k,\lambda)$-BIBD, then \[\frac{3 + \sqrt{4k+1}}{2} \leq \ell \leq \frac{k}{2}+1.\]
\end{theorem}

\begin{proof}
    From Lemma~\ref{lem: k=0 mod l-1, bounds on k}, $2(\ell-1) \leq k \leq (\ell-1)(\ell-2)$. Rearranging $k\leq (\ell-1)(\ell-2)$ gives \[\ell^2 -3\ell +2 - k \geq 0.\] Applying the quadratic formula and eliminating the negative solution for $\ell$ yields \[\ell \geq \frac{3 + \sqrt{4k+1}}{2}.\] Rearranging the lower bound $2(\ell-1)\leq k$ gives \[\ell \leq \frac{k}{2} + 1. \qedhere\] 
\end{proof}

It follows from Lemma~\ref{lem: k=0 mod l-1, bounds on k} and Theorem~\ref{thm: 0-ULSE bounds on ell} that in any $0$-ULSE $\ell$-coloured $(v,k,\lambda)$-BIBD, $\ell \geq 4$ and $k\geq 6$. Observe that the necessary conditions of Lemma~\ref{lem: k=0 mod l-1, bounds on k} are violated when $\ell \leq 3$. Likewise, the necessary conditions of Theorem~\ref{thm: 0-ULSE bounds on ell} are violated when $k\leq 5$.

The following result shows that within the set of all blocks $\mathcal{B}_j$ that do not contain any points from the $j$-th colour class, for any point $x$ not in the $j$-th colour class, $x$ is contained in a constant number of blocks in $\mathcal{B}_j$.

\begin{lemma}\label{lem: points in colour class are 'evenly' distributed}
    Let $(V,\mathcal{B})$ be a $0$-ULSE $\ell$-coloured $(v,k,\lambda)$-BIBD with colour classes $C_1$, $\dots$, $C_\ell$. For each $1\leq j\leq \ell$, let $\mathcal{B}_j = \{B\in \mathcal{B} \mid B\cap C_j = \emptyset\}$. If $i\neq j$ and $x\in C_i$, then $x$ occurs in exactly \[\frac{\lambda(v-1)}{k-1} - \frac{\lambda v(\ell-1)}{k\ell}\] blocks in $\mathcal{B}_j$. 
\end{lemma}

\begin{proof}
    Fix $x\in C_i$. Since every pair of points occurs in $\lambda$ blocks and $\mathcal{B}\backslash\mathcal{B}_j$ contains every block that contains both $x$ and points from $C_j$, we count the size of the following multiset: \[\big|\big\{y \mid \{x,y\}\subset B \in \mathcal{B}\backslash\mathcal{B}_j, y\in C_j \big\}\big| = \lambda|C_j| = \frac{\lambda v}{\ell}.\] Because every block containing $x$ and a point from $C_j$ contains $\frac{k}{\ell-1}$ pairs $xy$ where $y\in C_j$, the number of blocks in $\mathcal{B}\backslash\mathcal{B}_j$ containing $x$ is \[\frac{\lambda v}{\ell} \left( \frac{\ell-1}{k} \right) = \frac{\lambda v (\ell-1)}{k\ell}.\] Note that the replication number of $x$ in the BIBD is $\frac{\lambda(v-1)}{k-1}$. Therefore $x$ is contained in exactly \[\frac{\lambda(v-1)}{k-1} - \frac{\lambda v(\ell-1)}{k\ell}\] blocks in $\mathcal{B}_j$. 
\end{proof}

A direct consequence of Lemma~\ref{lem: points in colour class are 'evenly' distributed} is that $\frac{r}{\ell-1}\in \mathbb{Z}$, which we will now show.

\begin{theorem}
    If there exists a $0$-ULSE $\ell$-coloured $(v,k,\lambda)$-BIBD with replication number $r$, then $\frac{r}{\ell-1}\in \mathbb{Z}$.
\end{theorem}

\begin{proof}
    Let $\mathcal{D}$ be a $0$-ULSE $\ell$-coloured BIBD with $\mathcal{B}$ denoting the set of blocks of $\mathcal{D}$. Let $i,j \in \{1,\dots, \ell\}$ with $i\neq j$. Let $C_i$ and $C_j$ be two colour classes and let $\mathcal{B}_j = \{B\in \mathcal{B} \mid B\cap C_j = \emptyset\}$.

    From Lemma~\ref{lem: points in colour class are 'evenly' distributed}, each point in $C_i$ occurs in exactly $r - \frac{\lambda v(\ell-1)}{k\ell}$ blocks in $\mathcal{B}_j$. Furthermore, there are $\ell-1$ options for $j$ and there are $r$ blocks in $\mathcal{D}$ containing $x$. Therefore \[r = (\ell-1) \left( r - \frac{\lambda v(\ell-1)}{k\ell} \right)\] and so \[\frac{r}{\ell-1} = r - \frac{\lambda v(\ell-1)}{k\ell}.\] Since $r - \frac{\lambda v(\ell-1)}{k\ell}\in \mathbb{Z}$ by Lemma~\ref{lem: points in colour class are 'evenly' distributed}, $\frac{r}{\ell-1} \in \mathbb{Z}$. 
\end{proof}

In every BIBD that admits a $0$-ULSE colouring, its colour classes induce a design.

\begin{theorem}
    If $\mathcal{D}$ is a $0$-ULSE $\ell$-coloured $(v,k,\lambda)$-BIBD, then there exists a $\left(\frac{v}{\ell}, \frac{k}{\ell-1}, \lambda\right)$-BIBD.
\end{theorem}

\begin{proof}
    Let $B_1, \dots, B_b$ be the blocks of $\mathcal{D}$, let $C$ be one of the colour classes, and let \[\mathcal{B}^\prime = \{C\cap B_i \mid 1\leq i\leq b, C\cap B_i \neq \emptyset\}.\] From Lemma~\ref{lem: l-1 divides k, colour classes the same size, calculating alpha}, $|C| = \frac{v}{\ell}$. Consider the design $(C, \mathcal{B}^\prime)$. This design has $\frac{v}{\ell}$ points and blocks of size $|C\cap B_i| = \frac{k}{\ell-1}$. For any two distinct points $x,y\in C$, the pair $xy$ occurs in exactly $\lambda$ blocks in $\mathcal{D}$. Thus the pair $xy$ occurs exactly $\lambda$ times in $\mathcal{B}$. This proves that $(C,\mathcal{B}^\prime)$ is a $\left(\frac{v}{\ell}, \frac{k}{\ell-1}, \lambda\right)$-BIBD.
\end{proof}

\begin{theorem}\label{thm: symmetric 0-ULSE lambda value}
    If there exists a symmetric $0$-ULSE $\ell$-coloured $(v,k,\lambda)$-BIBD then $\lambda = \frac{k((\ell-1)^2 - k)}{(\ell-1)^2}$ and $\gamma = \frac{v}{\ell}$.
\end{theorem}

\begin{proof}
    In a symmetric BIBD, $k = r = \frac{\lambda(v-1)}{k-1}$. By rearranging and substituting in Equation~\eqref{eq: k=0 mod l-1, soln for v} we obtain 
    \begin{equation}\label{eq: lambda symmetric formula}
    \lambda = \frac{k(k-1)((\ell-1)^2 - k)}{\ell(\ell-2)k - (\ell-1)^2 + k} = \frac{k(k-1)((\ell-1)^2 - k)}{(k-1)(\ell-1)^2} = \frac{k((\ell-1)^2 - k)}{(\ell-1)^2}.
    \end{equation}

    Furthermore, in a symmetric BIBD, $v=b$. Substituting into Equation~\eqref{eq: gamma formula} and then applying Equations~\eqref{eq: k=0 mod l-1, soln for v} and \eqref{eq: lambda symmetric formula} yields
    \begin{align*}
        \gamma &= v - \frac{k^2}{\lambda} + \frac{(\ell-1)k^2}{\lambda k} - \frac{(\ell-1)k}{k} \\
        &= \frac{(\ell-2)k}{(\ell-1)^2 - k} \\
        &= \frac{v}{\ell}. \qedhere
    \end{align*}
\end{proof}

\begin{corollary}\label{cor: bds on lambda in terms of k and ell}
    If there exists a $0$-ULSE $\ell$-coloured $(v,k,\lambda)$-BIBD, then $\lambda \geq \frac{k((\ell-1)^2 - k)}{(\ell-1)^2}$.
\end{corollary}

\begin{proof}
    Since $\lambda = \frac{k((\ell-1)^2 - k)}{(\ell-1)^2}$ when the $0$-ULSE $\ell$-coloured $(v,k,\lambda)$-BIBD is symmetric, then choosing $\lambda < \frac{k((\ell-1)^2 - k)}{(\ell-1)^2}$ will contradict Lemma~\ref{lem: Fisher's ineq.}.
\end{proof}

Following a conference presentation given by Kellough in December 2025 on the preliminary contents of this paper, Brett Stevens posed the following question: Is the smallest possible candidate for $\lambda$ in a $0$-ULSE $\ell$-coloured $(v,k,\lambda)$-BIBD always the same as the value for $\lambda$ that would make such a BIBD symmetric? We will answer this question in the affirmative. By Theorem~\ref{thm: symmetric 0-ULSE lambda value}, we know the value of $\lambda$ that makes a $0$-ULSE $\ell$-coloured $(v,k,\lambda)$-BIBD symmetric is $\lambda = \frac{k((\ell-1)^2 - k)}{(\ell-1)^2}$. First, we show that $\frac{k((\ell-1)^2 - k)}{(\ell-1)^2}$ is always a positive integer. 

\begin{lemma}\label{lem: lower bd on lambda is always an int}
    If there exists a $0$-ULSE $\ell$-coloured $(v,k,\lambda)$-BIBD, then $\frac{k((\ell-1)^2 - k)}{(\ell-1)^2} \in \mathbb{Z}^+$.
\end{lemma}

\begin{proof}
    By the definition of uniform, $\frac{k}{\ell-1}\in \mathbb{Z}^+$. So $\frac{(\ell-1)^2 - k}{\ell-1}\in \mathbb{Z}$. Since $k\leq (\ell-1)(\ell-2)$ by Lemma~\ref{lem: k=0 mod l-1, bounds on k}, \[\frac{(\ell-1)^2 - k}{\ell-1} \geq \frac{(\ell-1)^2 - (\ell-1)(\ell-2)}{\ell-1} = 1.\] Therefore $\frac{k((\ell-1)^2 - k)}{(\ell-1)^2} \in \mathbb{Z}^+$.
\end{proof}

\begin{theorem}\label{thm: smallest lambda in 0-ULSE colouring}
    Let $v, k, \ell\in \mathbb{Z}^+$ satisfy Equation~\eqref{eq: k=0 mod l-1, soln for v}. Then the smallest index $\lambda_{\min}$ such that Lemmas \ref{lem: replication num of a BIBD} and \ref{lem: num of blocks in a BIBD} are satisfied for a BIBD on $v$ points and block size $k$ with a $0$-ULSE $\ell$-colouring is $\frac{k(k-1)}{v-1} = \frac{k((\ell-1)^2 - k)}{(\ell-1)^2}$. 
\end{theorem}

\begin{proof}
    It can be easily checked that $\lambda = \frac{k(k-1)}{v-1} = \frac{k((\ell-1)^2 - k)}{(\ell-1)^2}$ satisfies Lemmas \ref{lem: replication num of a BIBD} and \ref{lem: num of blocks in a BIBD} for a $(v,k,\lambda)$-BIBD. Suppose $\lambda_{\min}$ is the smallest integer described in the statement of the theorem. From Corollary~\ref{cor: bds on lambda in terms of k and ell}, $\lambda_{\min} \geq \frac{k((\ell-1)^2 - k)}{(\ell-1)^2}$. From Lemma~\ref{lem: lower bd on lambda is always an int}, the lower bound from Corollary~\ref{cor: bds on lambda in terms of k and ell} is always an integer. Therefore, $\frac{k((\ell-1)^2 - k)}{(\ell-1)^2}$ is the smallest possible value for the index of a $0$-ULSE coloured BIBD. So $\lambda_{\min} = \frac{k((\ell-1)^2 - k)}{(\ell-1)^2}$ as required.
\end{proof}

By the work done in the proof of Lemma~\ref{lem: lower bd on lambda is always an int}, $\frac{(\ell-1)^2 - k}{\ell-1} \in \mathbb{Z}^+$. Furthermore, if $\frac{k}{\ell - 1} = 1$, then a $0$-ULSE $\ell$-colouring of a $(v,k,\lambda)$-BIBD is trivial. Thus, in any nontrivial $0$-ULSE $\ell$-colouring, $\frac{k}{\ell - 1} \geq 2$. Since $\lambda \geq \frac{k((\ell-1)^2 - k)}{(\ell-1)^2}$ in any $0$-ULSE coloured $(v,k,\lambda)$-BIBD by Theorem~\ref{thm: smallest lambda in 0-ULSE colouring}, it follows that $\lambda \geq 2$ in any nontrivial $0$-ULSE coloured BIBD. 

To finish this section, we show that $0$-ULSE colourings are unique in the sense that if a BIBD $\mathcal{D}$ admits a nontrivial $0$-ULSE $\ell$-colouring, then there is no $\ell^\prime \in \mathbb{Z}^+\backslash\{\ell\}$ such that $\mathcal{D}$ also admits a nontrivial $0$-ULSE $\ell^\prime$-colouring. 

\begin{theorem}\label{thm: uniqueness}
    If a $(v,k,\lambda)$-BIBD exists, then there is at most one $\ell\in \mathbb{Z}^+$ such that a nontrivial $0$-ULSE $\ell$-colouring of the BIBD exists. Furthermore, if $\mathcal{D}_1$ is a $(v,k,\lambda_1)$-BIBD with a nontrivial $0$-ULSE $\ell_1$-colouring and $\mathcal{D}_2$ is a $(v,k,\lambda_2)$-BIBD with a nontrivial $0$-ULSE $\ell_2$-colouring, then $\ell_1 = \ell_2$. 
\end{theorem}

\begin{proof}
    Fix $v,k\in \mathbb{Z}^+$ such that a $(v,k,\lambda)$-BIBD exists for some $\lambda\in \mathbb{Z}^+$. Suppose there exists a $0$-ULSE $\ell$-colouring of this BIBD. By Theorem~\ref{thm: strong necessary condition for 0-LSE ell-col with ell-1 | k}, \[v = \frac{\ell(\ell-2)k}{(\ell-1)^2 - k}.\] Rearranging gives \[(v-k)\ell^2 - 2(v-k)\ell - v(k-1) = 0.\] Applying the quadratic formula and simplifying yields 
    \begin{equation}\label{eq: ell in terms of vk}
        \ell = 1\pm \frac{2\sqrt{k(v-1)(v-k)}}{2(v-k)}.
    \end{equation}
    Since $k\in \mathbb{Z}^+$, \[1 - \frac{2\sqrt{k(v-1)(v-k)}}{2(v-k)} \leq 1 - \frac{2\sqrt{(v-k)(v-k)}}{2(v-k)} = 1-1 = 0\] which is not an admissible number of colours. Therefore we have at most one solution for $\ell$.

    Notably, our proof works independently of the value of $\lambda$. Therefore for any two BIBDs on the same number of points with the same block size, there is at most one $\ell$ such that they each have a $0$-ULSE $\ell$-colouring. In particular, $\mathcal{D}_1$ and $\mathcal{D}_2$ cannot have a $0$-ULSE $\ell_1$-colouring and a $0$-ULSE $\ell_2$-colouring respectively with $\ell_1 \neq \ell_2$. 
\end{proof}

As a consequence of the calculations done in the proof of Theorem~\ref{thm: uniqueness}, we have the following.

\begin{corollary}
    If a $0$-ULSE $\ell$-coloured $(v,k,\lambda)$-BIBD exists, then $k(v-1)(v-k)$ is a perfect square.
\end{corollary}

\begin{proof}
    Equation~\eqref{eq: ell in terms of vk} gives a formula for $\ell$ in terms of $v$ and $k$. Since $\ell\in \mathbb{Z}^+$, \[1 + \frac{2\sqrt{k(v-1)(v-k)}}{2(v-k)}\] is also an integer. Therefore $k(v-1)(v-k)$ is a perfect square.
\end{proof}

\section{A General Construction}\label{sec: Construction}

In this section we give a framework for constructing $0$-ULSE coloured BIBDs. This framework requires two designs as input ingredients: resolvable BIBDs and transversal designs.

A BIBD $(V, \mathcal{B})$ is \emph{resolvable} if its blocks can be partitioned into sets $\Pi_1, \dots, \Pi_r$ such that for each $1\leq i\leq r$, $\cup_{B \in \Pi_i} B = V$ and $B_1 \cap B_2 = \emptyset$ for any distinct $B_1, B_2 \in \Pi_i$. The sets $\Pi_1, \dots, \Pi_r$ are called \emph{resolution classes}. We abbreviate resolvable BIBD to RBIBD.

If $k,n,\lambda \in \mathbb{Z}^+$ and $n > 2$, then a \emph{transversal design} TD$_\lambda(k,n)$ is a triple $(V, \mathcal{G}, \mathcal{B})$ where $V$ is a set of $kn$ points, $\mathcal{G}$ is a partition of $V$ into $k$ $n$-element subsets called \emph{groups}, $\mathcal{B}$ is a collection of $k$-element subsets, called \emph{blocks}, that contain exactly one point from each group, and for every pair of distinct points $x,y \in V$ that are not in the same group, there are exactly $\lambda$ blocks containing both $x$ and $y$. In particular, no two points in the same group are in the same block. If $\lambda = 1$, then we write TD$(k,n)$ instead of TD$_1(k,n)$. The following result follows from simple counting arguments.

\begin{lemma}\label{lem: r and b for TDs}
    If a TD$_\lambda (k,n)$ exists, then each point occurs in exactly $\lambda n$ blocks and there are exactly $\lambda n^2$ blocks.
\end{lemma}

Transversal designs and resolvable BIBDs can be used to construct $0$-ULSE coloured BIBDs if the parameters are chosen in a specific way. The following is the main result of this paper.

\begin{theorem}\label{thm: general construction for nice 0-LSE l-coloured BIBDs}
    Let $v,k,\lambda \in \mathbb{Z}^+$ and $\ell\geq 4$. If there exist a TD$_{\frac{\lambda}{\ell-2}}\left(\ell-1, \frac{v(\ell-1)}{k\ell}\right)$ and a $\left(\frac{v}{\ell}, \frac{k}{\ell-1}, \frac{k\ell(\ell-2)}{v(\ell-1)}\right)$-RBIBD, then there exists a $0$-ULSE $\ell$-coloured $(v,k,\lambda)$-BIBD.
\end{theorem}

In advance of proving Theorem~\ref{thm: general construction for nice 0-LSE l-coloured BIBDs}, we illustrate the proof techniques by constructing a small example: a $0$-ULSE $4$-coloured $(16,6,2)$-BIBD. The existence of such a BIBD proves that the bounds in Theorem~\ref{thm: 0-ULSE bounds on ell} are tight.

We begin with a $(4,3,2)$-BIBD, which is illustrated in Figure~\ref{fig: (4,3,2)-BIBD}. The points in this $(4,3,2)$-BIBD will act as the colours at the end of this construction. Let $D_0$, $D_1$, $D_2$, and $D_3$ be the blocks of the $(4,3,2)$-BIBD where the blocks are labelled from top to bottom in Figure~\ref{fig: (4,3,2)-BIBD}.

\begin{figure}
    \begin{center}
\begin{tikzpicture}
        \tikzstyle{circ}=[circle, very thick, fill=red!25, draw=red, minimum size=15pt,inner sep=0pt]
        \tikzstyle{square} = [draw=PineGreen, fill=green!25, regular polygon, regular polygon sides=4, very thick, minimum size=18pt, inner sep=0pt]
        \tikzstyle{pent} = [draw=blue, fill=blue!25, regular polygon, regular polygon sides=5, very thick, minimum size=15pt, inner sep=0pt]
        \tikzstyle{hex} = [draw=Bittersweet, fill=YellowOrange!25, regular polygon, regular polygon sides=6, very thick, minimum size=15pt, inner sep=0pt]

        \node at (-3.5,-0.75) {Points:};
        \node[circ] at (-5,-1.5) {};
        \node[square] at (-4,-1.5) {};
        \node[pent] at (-3,-1.5) {};
        \node[hex] at (-2,-1.5) {};

        \node at (2,0.75) {Blocks:};

        \node[square, xshift=1cm] at (0,0) {};
        \node[pent, xshift=1cm] at (1,0) {};
        \node[hex, xshift=1cm] at (2,0) {};

        \node[pent, xshift=1cm] at (0,-0.75) {};
        \node[hex, xshift=1cm] at (1,-0.75) {};
        \node[circ, xshift=1cm] at (2,-0.75) {};

        \node[hex, xshift=1cm] at (0,-1.5) {};
        \node[circ, xshift=1cm] at (1,-1.5) {};
        \node[square, xshift=1cm] at (2,-1.5) {};

        \node[circ, xshift=1cm] at (0,-2.25) {};
        \node[square, xshift=1cm] at (1,-2.25) {};
        \node[pent, xshift=1cm] at (2,-2.25) {};

        \node at (0.5,0) {\large $\{$};
        \node at (3.5,0) {\large $\}$};

        \node at (0.5,-0.75) {\large $\{$};
        \node at (3.5,-0.75) {\large $\}$};

        \node at (0.5,-1.5) {\large $\{$};
        \node at (3.5,-1.5) {\large $\}$};

        \node at (0.5,-2.25) {\large $\{$};
        \node at (3.5,-2.25) {\large $\}$};

        \node at (1.5,-0.25) {,};
        \node at (2.5,-0.25) {,};

        \node at (1.5,-1) {,};
        \node at (2.5,-1) {,};

        \node at (1.5,-1.75) {,};
        \node at (2.5,-1.75) {,};

        \node at (1.5,-2.5) {,};
        \node at (2.5,-2.5) {,};
        \end{tikzpicture}
\end{center}
    \caption{A $(4,3,2)$-BIBD.}
    \label{fig: (4,3,2)-BIBD}
\end{figure}
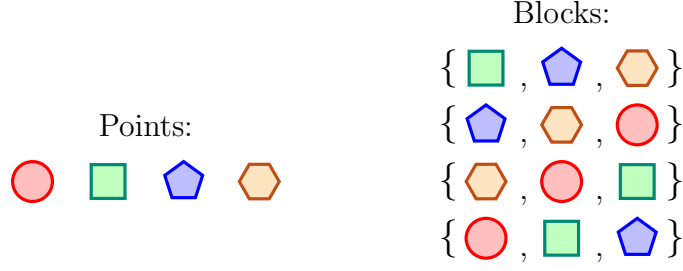

\begin{table}[ht]
    \begin{center}
        \begin{tabular}{c|c|c}
            $\Pi_1$ & $\Pi_2$ & $\Pi_3$ \\
            \hline
            $A_1 = \{0,1\}$ & $A_3 = \{0,2\}$ & $A_5 = \{0,3\}$ \\
            $A_2 = \{2,3\}$ & $A_4 = \{1,3\}$ & $A_6 = \{1,2\}$
        \end{tabular}
    \end{center}
    \caption{A $(4,2,1)$-RBIBD with resolution classes $\Pi_1$, $\Pi_2$, and $\Pi_3$.}
    \label{tab: (4,2,1)-RBIBD}
\end{table}

Next, we consider the $(4,2,1)$-RBIBD with its blocks listed in Table~\ref{tab: (4,2,1)-RBIBD}. Note that the $(4,2,1)$-RBIBD has three resolution classes; this number coincides with the size of each $D_i$. Thus, for each $D_i$, we can assign a resolution class to each point in $D_i$. We will make this assignment based on the order of the elements in each $D_i$ as displayed in Figure~\ref{fig: (4,3,2)-BIBD}. In $D_0$ the square is paired with the resolution class $\Pi_1 = \{A_1,A_2\}$, the pentagon is paired with the resolution class $\Pi_2 = \{A_3,A_4\}$, and the hexagon is paired with the resolution class $\Pi_3 = \{A_5,A_6\}$. We repeat this similarly for the other $D_i$. For every \{shape, resolution class\} pair, we create two new points which are themselves pairs: \{shape, $A_i$\} and \{shape, $A_j$\} where $\{A_i, A_j\}$ is the resolution class. This produces $24$ points which are illustrated in Figure~\ref{fig: BIBD + RBIBD}.

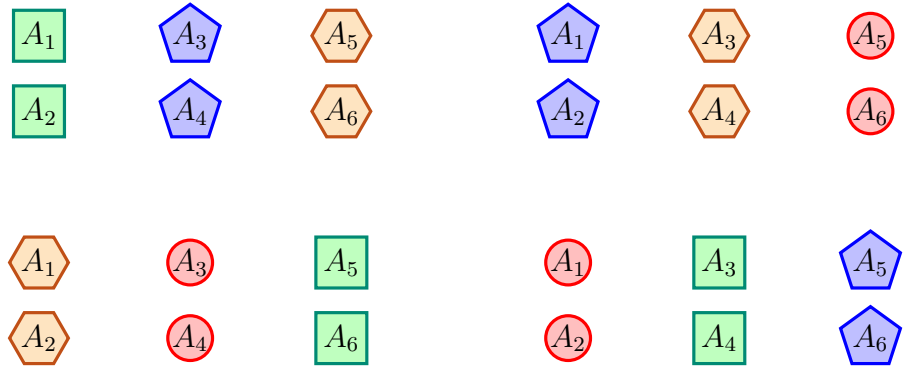
\begin{figure}
    \begin{center}
\begin{tikzpicture}
        \tikzstyle{circ}=[circle, very thick, fill=red!25, draw=red, minimum size=15pt,inner sep=0pt]
        \tikzstyle{square} = [draw=PineGreen, fill=green!25, regular polygon, regular polygon sides=4, very thick, minimum size=15pt, inner sep=0pt]
        \tikzstyle{pent} = [draw=blue, fill=blue!25, regular polygon, regular polygon sides=5, very thick, minimum size=15pt, inner sep=0pt]
        \tikzstyle{hex} = [draw=Bittersweet, fill=YellowOrange!25, regular polygon, regular polygon sides=6, very thick, minimum size=15pt, inner sep=0pt]

        \node[square] at (1,-2) {$A_1$};
        \node[square] at (1,-3) {$A_2$};
        \node[pent] at (3,-2) {$A_3$};
        \node[pent] at (3,-3) {$A_4$};
        \node[hex] at (5,-2) {$A_5$};
        \node[hex] at (5,-3) {$A_6$};

        \node[pent] at (8,-2) {$A_1$};
        \node[pent] at (8,-3) {$A_2$};
        \node[hex] at (10,-2) {$A_3$};
        \node[hex] at (10,-3) {$A_4$};
        \node[circ] at (12,-2) {$A_5$};
        \node[circ] at (12,-3) {$A_6$};

        \node[hex] at (1,-5) {$A_1$};
        \node[hex] at (1,-6) {$A_2$};
        \node[circ] at (3,-5) {$A_3$};
        \node[circ] at (3,-6) {$A_4$};
        \node[square] at (5,-5) {$A_5$};
        \node[square] at (5,-6) {$A_6$};

        \node[circ] at (8,-5) {$A_1$};
        \node[circ] at (8,-6) {$A_2$};
        \node[square] at (10,-5) {$A_3$};
        \node[square] at (10,-6) {$A_4$};
        \node[pent] at (12,-5) {$A_5$};
        \node[pent] at (12,-6) {$A_6$};
        \end{tikzpicture}
\end{center}
    \caption{Combining points from the $(4,3,2)$-BIBD and the $(4,2,1)$-RBIBD.}
    \label{fig: BIBD + RBIBD}
\end{figure}

For each set of six points that are formed from a block of the $(4,3,2)$-BIBD, we form a TD$_1(3,2)$ where each shape gets its own group. We form the blocks of each TD as shown in Figure~\ref{fig: TD_1(3,2)}.

\begin{figure}
    \begin{center}
\begin{tikzpicture}
        \tikzstyle{circ}=[circle, very thick, fill=red!25, draw=red, minimum size=15pt,inner sep=0pt]
        \tikzstyle{square} = [draw=PineGreen, fill=green!25, regular polygon, regular polygon sides=4, very thick, minimum size=15pt, inner sep=0pt]
        \tikzstyle{pent} = [draw=blue, fill=blue!25, regular polygon, regular polygon sides=5, very thick, minimum size=15pt, inner sep=0pt]
        \tikzstyle{hex} = [draw=Bittersweet, fill=YellowOrange!25, regular polygon, regular polygon sides=6, very thick, minimum size=15pt, inner sep=0pt]

        \draw[ultra thick, black, dotted] (0,-2)--(6,-2);
        \draw[ultra thick, red] (0,-1.5)--(3,-3)--(6,-3);
        \draw[ultra thick, cyan, dashed] (0,-3)--(3,-3)--(6,-1.5);
        \draw[ultra thick, orange, dashdotted] (0,-3.5)--(3,-2)--(6,-3.5);

        \node[square] at (1,-2) {$A_1$};
        \node[square] at (1,-3) {$A_2$};
        \node[pent] at (3,-2) {$A_3$};
        \node[pent] at (3,-3) {$A_4$};
        \node[hex] at (5,-2) {$A_5$};
        \node[hex] at (5,-3) {$A_6$};

        \draw[ultra thick, black, dotted] (7,-2)--(13,-2);
        \draw[ultra thick, red] (7,-1.5)--(10,-3)--(13,-3);
        \draw[ultra thick, cyan, dashed] (7,-3)--(10,-3)--(13,-1.5);
        \draw[ultra thick, orange, dashdotted] (7,-3.5)--(10,-2)--(13,-3.5);

        \node[pent] at (8,-2) {$A_1$};
        \node[pent] at (8,-3) {$A_2$};
        \node[hex] at (10,-2) {$A_3$};
        \node[hex] at (10,-3) {$A_4$};
        \node[circ] at (12,-2) {$A_5$};
        \node[circ] at (12,-3) {$A_6$};

        \draw[ultra thick, black, dotted] (0,-5)--(6,-5);
        \draw[ultra thick, red] (0,-4.5)--(3,-6)--(6,-6);
        \draw[ultra thick, cyan, dashed] (0,-6)--(3,-6)--(6,-4.5);
        \draw[ultra thick, orange, dashdotted] (0,-6.5)--(3,-5)--(6,-6.5);

        \node[hex] at (1,-5) {$A_1$};
        \node[hex] at (1,-6) {$A_2$};
        \node[circ] at (3,-5) {$A_3$};
        \node[circ] at (3,-6) {$A_4$};
        \node[square] at (5,-5) {$A_5$};
        \node[square] at (5,-6) {$A_6$};

        \draw[ultra thick, black, dotted] (7,-5)--(13,-5);
        \draw[ultra thick, red] (7,-4.5)--(10,-6)--(13,-6);
        \draw[ultra thick, cyan, dashed] (7,-6)--(10,-6)--(13,-4.5);
        \draw[ultra thick, orange, dashdotted] (7,-6.5)--(10,-5)--(13,-6.5);

        \node[circ] at (8,-5) {$A_1$};
        \node[circ] at (8,-6) {$A_2$};
        \node[square] at (10,-5) {$A_3$};
        \node[square] at (10,-6) {$A_4$};
        \node[pent] at (12,-5) {$A_5$};
        \node[pent] at (12,-6) {$A_6$};
        \end{tikzpicture}
\end{center}
    \caption{Forming transversal designs from the points constructed in Figure~\ref{fig: BIBD + RBIBD}.}
    \label{fig: TD_1(3,2)}
\end{figure}
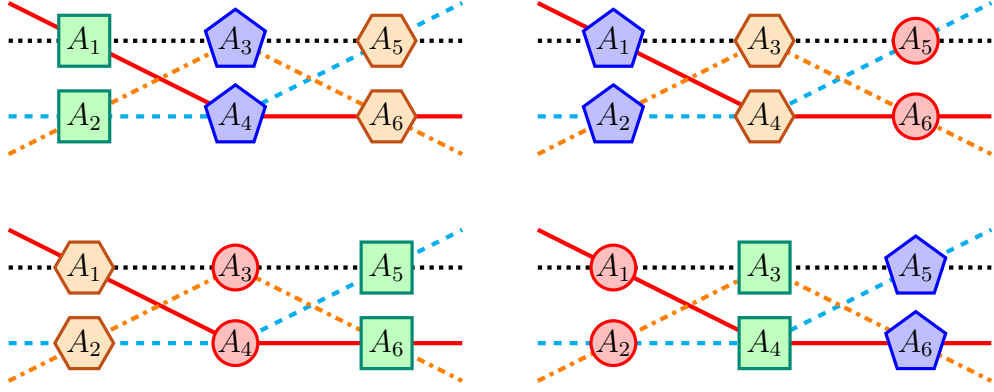

The construction ends by taking each block in Figure~\ref{fig: TD_1(3,2)} and expanding it into one of size six in the following way. For every block in Figure~\ref{fig: TD_1(3,2)} and every point \{shape, $A_i$\} in the block, we replace \{shape, $A_i$\} with two points: \{shape, $a_1$\} and \{shape, $a_2$\} where $A_i = \{a_1, a_2\}$. The resulting $(16,6,2)$-BIBD is shown in Figure~\ref{fig: 0-ULSE 4-col (16,6,2)-BIBD} and this concludes the example.

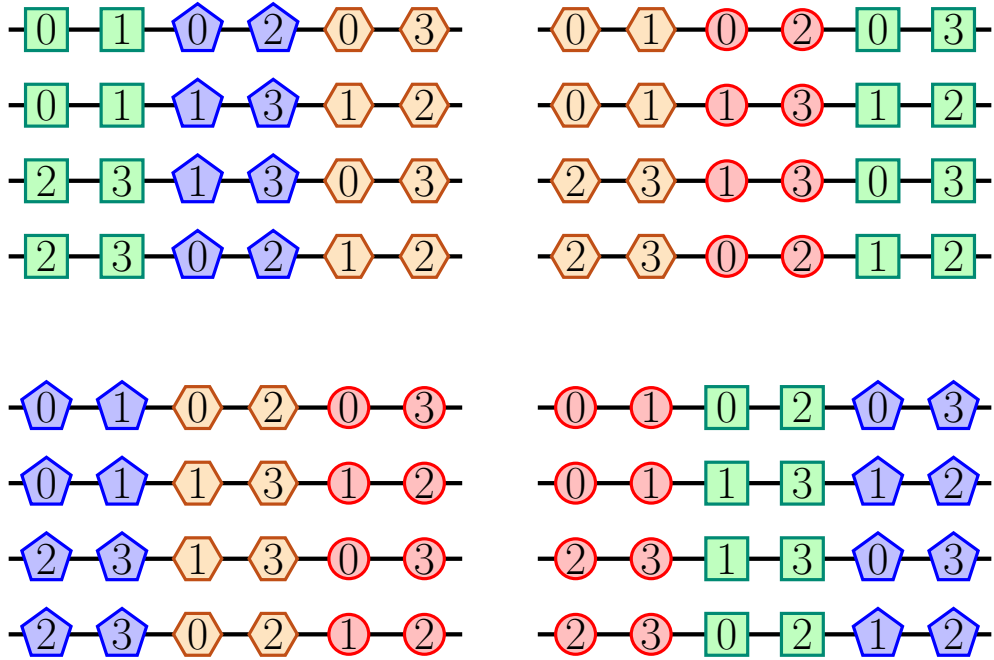
\begin{figure}
    \begin{center}
    \begin{tikzpicture}[scale=1]
        \tikzstyle{circ}=[circle, very thick, fill=red!25, draw=red, minimum size=15pt,inner sep=0pt]
        \tikzstyle{square} = [draw=PineGreen, fill=green!25, regular polygon, regular polygon sides=4, very thick, minimum size=15pt, inner sep=0pt]
        \tikzstyle{pent} = [draw=blue, fill=blue!25, regular polygon, regular polygon sides=5, very thick, minimum size=15pt, inner sep=0pt]
        \tikzstyle{hex} = [draw=Bittersweet, fill=YellowOrange!25, regular polygon, regular polygon sides=6, very thick, minimum size=15pt, inner sep=0pt]

        \foreach \x in {-0.5, 6.5}{
            \foreach \y in {0, -1, -2, -3, -5, -6, -7, -8}{
                \draw[ultra thick] (\x, \y)--(\x + 6, \y);
            }
        }

        \node[square] at (0,0) {\Large$0$};
        \node[square] at (1,0) {\Large$1$};
        \node[pent] at (2,0) {\Large$0$};
        \node[pent] at (3,0) {\Large$2$};
        \node[hex] at (4,0) {\Large$0$};
        \node[hex] at (5,0) {\Large$3$};

        \node[square] at (0,-1) {\Large$0$};
        \node[square] at (1,-1) {\Large$1$};
        \node[pent] at (2,-1) {\Large$1$};
        \node[pent] at (3,-1) {\Large$3$};
        \node[hex] at (4,-1) {\Large$1$};
        \node[hex] at (5,-1) {\Large$2$};

        \node[square] at (0,-2) {\Large$2$};
        \node[square] at (1,-2) {\Large$3$};
        \node[pent] at (2,-2) {\Large$1$};
        \node[pent] at (3,-2) {\Large$3$};
        \node[hex] at (4,-2) {\Large$0$};
        \node[hex] at (5,-2) {\Large$3$};

        \node[square] at (0,-3) {\Large$2$};
        \node[square] at (1,-3) {\Large$3$};
        \node[pent] at (2,-3) {\Large$0$};
        \node[pent] at (3,-3) {\Large$2$};
        \node[hex] at (4,-3) {\Large$1$};
        \node[hex] at (5,-3) {\Large$2$};

        \node[pent] at (0,-5) {\Large$0$};
        \node[pent] at (1,-5) {\Large$1$};
        \node[hex] at (2,-5) {\Large$0$};
        \node[hex] at (3,-5) {\Large$2$};
        \node[circ] at (4,-5) {\Large$0$};
        \node[circ] at (5,-5) {\Large$3$};

        \node[pent] at (0,-6) {\Large$0$};
        \node[pent] at (1,-6) {\Large$1$};
        \node[hex] at (2,-6) {\Large$1$};
        \node[hex] at (3,-6) {\Large$3$};
        \node[circ] at (4,-6) {\Large$1$};
        \node[circ] at (5,-6) {\Large$2$};

        \node[pent] at (0,-7) {\Large$2$};
        \node[pent] at (1,-7) {\Large$3$};
        \node[hex] at (2,-7) {\Large$1$};
        \node[hex] at (3,-7) {\Large$3$};
        \node[circ] at (4,-7) {\Large$0$};
        \node[circ] at (5,-7) {\Large$3$};

        \node[pent] at (0,-8) {\Large$2$};
        \node[pent] at (1,-8) {\Large$3$};
        \node[hex] at (2,-8) {\Large$0$};
        \node[hex] at (3,-8) {\Large$2$};
        \node[circ] at (4,-8) {\Large$1$};
        \node[circ] at (5,-8) {\Large$2$};


        \node[hex] at (7,0) {\Large$0$};
        \node[hex] at (8,0) {\Large$1$};
        \node[circ] at (9,0) {\Large$0$};
        \node[circ] at (10,0) {\Large$2$};
        \node[square] at (11,0) {\Large$0$};
        \node[square] at (12,0) {\Large$3$};

        \node[hex] at (7,-1) {\Large$0$};
        \node[hex] at (8,-1) {\Large$1$};
        \node[circ] at (9,-1) {\Large$1$};
        \node[circ] at (10,-1) {\Large$3$};
        \node[square] at (11,-1) {\Large$1$};
        \node[square] at (12,-1) {\Large$2$};

        \node[hex] at (7,-2) {\Large$2$};
        \node[hex] at (8,-2) {\Large$3$};
        \node[circ] at (9,-2) {\Large$1$};
        \node[circ] at (10,-2) {\Large$3$};
        \node[square] at (11,-2) {\Large$0$};
        \node[square] at (12,-2) {\Large$3$};

        \node[hex] at (7,-3) {\Large$2$};
        \node[hex] at (8,-3) {\Large$3$};
        \node[circ] at (9,-3) {\Large$0$};
        \node[circ] at (10,-3) {\Large$2$};
        \node[square] at (11,-3) {\Large$1$};
        \node[square] at (12,-3) {\Large$2$};

        \node[circ] at (7,-5) {\Large$0$};
        \node[circ] at (8,-5) {\Large$1$};
        \node[square] at (9,-5) {\Large$0$};
        \node[square] at (10,-5) {\Large$2$};
        \node[pent] at (11,-5) {\Large$0$};
        \node[pent] at (12,-5) {\Large$3$};

        \node[circ] at (7,-6) {\Large$0$};
        \node[circ] at (8,-6) {\Large$1$};
        \node[square] at (9,-6) {\Large$1$};
        \node[square] at (10,-6) {\Large$3$};
        \node[pent] at (11,-6) {\Large$1$};
        \node[pent] at (12,-6) {\Large$2$};

        \node[circ] at (7,-7) {\Large$2$};
        \node[circ] at (8,-7) {\Large$3$};
        \node[square] at (9,-7) {\Large$1$};
        \node[square] at (10,-7) {\Large$3$};
        \node[pent] at (11,-7) {\Large$0$};
        \node[pent] at (12,-7) {\Large$3$};

        \node[circ] at (7,-8) {\Large$2$};
        \node[circ] at (8,-8) {\Large$3$};
        \node[square] at (9,-8) {\Large$0$};
        \node[square] at (10,-8) {\Large$2$};
        \node[pent] at (11,-8) {\Large$1$};
        \node[pent] at (12,-8) {\Large$2$};
    \end{tikzpicture}
\end{center}
    \caption{The $16$ blocks of a $0$-ULSE $4$-coloured $(16,6,2)$-BIBD.}
    \label{fig: 0-ULSE 4-col (16,6,2)-BIBD}
\end{figure}

We are now ready to prove Theorem~\ref{thm: general construction for nice 0-LSE l-coloured BIBDs}.

\begin{proof}[Proof of Theorem~\ref{thm: general construction for nice 0-LSE l-coloured BIBDs}]
    Let $P =\{c_0, \dots, c_{\ell-1}\}$ be a set of points and for each $0\leq m\leq \ell-1$, let $E_m=P\backslash\{c_m\}$. Note that the set of blocks $\{E_0, \dots, E_{\ell-1}\}$ form an $(\ell, \ell-1, \ell-2)$-BIBD which we will call $\mathcal{E}$.

    Let $\mathcal{R}$ be a $\left(\frac{v}{\ell}, \frac{k}{\ell-1}, \frac{k\ell(\ell-2)}{v(\ell-1)}\right)$-RBIBD. Then $\mathcal{R}$ contains \[b_{\mathcal{R}} := \frac{k\ell(\ell-2)}{v(\ell-1)}\left( \frac{\frac{v}{\ell}\left( \frac{v}{\ell} - 1\right)}{\frac{k}{\ell-1}\left( \frac{k}{\ell-1} - 1 \right)} \right) = \frac{(v-\ell)(\ell-1)(\ell-2)}{\ell(k-\ell+1)}\] blocks, resolution classes of size \[s_{\mathcal{R}} := \frac{v(\ell-1)}{k\ell},\] and $\frac{b_{\mathcal{R}}}{s_{\mathcal{R}}}$ resolution classes. If we substitute Equation~\eqref{eq: k=0 mod l-1, soln for v} into the above values for $b_{\mathcal{R}}$ and $s_{\mathcal{R}}$ we obtain \[b_{\mathcal{R}} = \frac{(\ell-1)^2(\ell-2)}{(\ell-1)^2 - k}\] and \[s_{\mathcal{R}} = \frac{(\ell-1)(\ell-2)}{(\ell-1)^2 - k}.\] Consequently, \[\frac{b_{\mathcal{R}}}{s_{\mathcal{R}}} = \ell -1.\]
    
    Let $\mathcal{A} = \{A_1, \dots, A_{b_{\mathcal{R}}}\}$ be the blocks of $\mathcal{R}$. Without loss of generality, assume that $\{A_i, A_{i+1}, \dots, A_{i + s_{\mathcal{R}} - 1}\}$ is a resolution class for each $i\equiv 1 \pmod{s_{\mathcal{R}}}$. For each $j\equiv i \pmod{s_{\mathcal{R}}}$ where $1\leq i\leq s_{\mathcal{R}} - 1$, let $A_j = \{a_j^{(i-1)\frac{k}{\ell-1} + 1}, \dots, a_j^{i\frac{k}{\ell-1}}\}$. For $j\equiv 0 \pmod{s_{\mathcal{R}}}$, let $A_j = \{a_j^{(s_{\mathcal{R}}-1)\frac{k}{\ell-1} + 1},\dots, a_j^{\frac{v}{\ell}}\}$. Note that if $a_{j_1}^{q_1} \in A_{j_1}$ and  $a_{j_2}^{q_2} \in A_{j_2}$, then it may be true that $a_{j_1}^{q_1} = a_{j_2}^{q_2}$ since blocks in $\mathcal{R}$ are not necessarily disjoint. In the case where $A_{j_1}$ and $A_{j_2}$ are in the same resolution class, then $a_{j_1}^{q_1} \neq a_{j_2}^{q_2}$.

    For each $0\leq m\leq \ell-1$, let $\mathcal{T}_m$ be a TD$_{\frac{\lambda}{\ell-2}}\left(\ell-1, \frac{v(\ell-1)}{k\ell}\right)$ constructed from the set of points \[\left\{ (c_{m+i},A_{(i-1)s_{\mathcal{R}} + j}) \mid 1\leq i \leq \ell-1, 1\leq j\leq s_{\mathcal{R}} \right\}\] with groups of the form $\{(c_{m+i},A_{(i-1)s_{\mathcal{R}} + 1}), \dots, (c_{m+i},A_{is_{\mathcal{R}}})\}$ where $1\leq i\leq \ell-1$ and addition in the subscripts of each $c_{m+i}$ is done modulo $\ell$. Note that the first coordinates of the points of $\mathcal{T}_m$ are precisely the points of $E_m$. From Lemma~\ref{lem: r and b for TDs} we know that such a design contains exactly $\frac{\lambda}{\ell-2}\left( \frac{v(\ell-1)}{k\ell} \right)^2 = \frac{\lambda s_{\mathcal{R}}^2}{\ell-2}$ blocks. For each $0\leq m\leq \ell-1$, let $\mathcal{B}_m$ be the set of $\frac{\lambda s_{\mathcal{R}}^2}{\ell-2}$ blocks of $\mathcal{T}_m$. 

    Now, for each $0\leq m\leq \ell-1$ and for each $B\in \mathcal{B}_m$, say $B = \{(c_{m+1},A_{j_0}), \dots, (c_{m + \ell - 1},A_{j_{\ell -2}})\}$, define 
    \begin{align*}
        B^\prime = \Big\{ &(c_{m+1}, a_{j_0}^{q_0}), \dots, (c_{m+1},a_{j_0}^{q_0 + \frac{k}{\ell-1} - 1}),  \\
        &(c_{m+2}, a_{j_1}^{q_1}), \dots, (c_{m+2},a_{j_1}^{q_1 + \frac{k}{\ell-1} - 1}), \\
        &\vdots\\
        &(c_{m+\ell-1}, a_{j_{\ell-2}}^{q_{\ell-2}}), \dots, (c_{m+\ell-1},a_{j_{\ell-2}}^{q_{\ell-2} + \frac{k}{\ell-1} - 1})\Big\}
    \end{align*}
    where for each $1\leq p\leq s_{\mathcal{R}} - 1$ and $j_t \equiv p \pmod{s_{\mathcal{R}}}$, $q_t = (p-1)\frac{k}{\ell-1}+1$, and for $j_t \equiv 0 \pmod{s_{\mathcal{R}}}$, $q_t = \left( s_\mathcal{R} - 1 \right)\frac{k}{\ell-1} + 1$. Let \[\mathcal{B}^\prime_m = \{B^\prime \mid B\in \mathcal{B}_m\}\] and let \[\mathcal{B}^\prime = \bigcup_{0\leq m\leq \ell-1} \mathcal{B}^\prime_m.\] Let \[X = \left\{ (c_i,a_j^q) \mid 0\leq i\leq \ell-1, a_j^q \text{ is a point in } \mathcal{R} \right\}.\] We claim that $(X, \mathcal{B}^\prime)$ is a $0$-ULSE $\ell$-colourable $(v,k,\lambda)$-BIBD.

    For any $(c_i,a_j^q) \in X$, there are $\ell$ options for $c_i$ and $\frac{v}{\ell}$ options for $a_j^q$. So $|X| = (\ell)\frac{v}{\ell} = v$. Each block contains $(\ell-1)(\frac{k}{\ell-1}) = k$ points. 
    
    Now, consider two distinct points $(c_{m_1},a_{j_1}^{q_1})$ and $(c_{m_2},a_{j_2}^{q_2})$. Suppose first, $m_1 \neq m_2$. There are $\ell-2$ blocks in $\mathcal{E}$ containing the pair $c_{m_1} c_{m_2}$. For any $0\leq m\leq \ell-1$ such that $E_m$ contains the pair $c_{m_1} c_{m_2}$, the transversal design $\mathcal{T}_m$ has index $\frac{\lambda}{\ell-2}$. Thus, the pair of points $(c_{m_1}, A_{j_1})$ and $(c_{m_2}, A_{j_2})$ appears in exactly $\frac{\lambda}{\ell-2}$ blocks of $\mathcal{T}_m$. Recall that the set of groups of $\mathcal{T}_m$ is \[\left\{ \{(c_{m+i}, A_{(i-1)s_{\mathcal{R}}+1}), (c_{m+i}, A_{(i-1)s_{\mathcal{R}}+2}) \dots, (c_{m+i}, A_{is_{\mathcal{R}}})\} \mid 1\leq i\leq \ell-1 \right\}\] where for each $1\leq i\leq \frac{b_\mathcal{R}}{s_\mathcal{R}} = \ell - 1$, $\{A_{(i-1)s_{\mathcal{R}}+1}, A_{(i-1)s_{\mathcal{R}}+2}, \dots, A_{is_{\mathcal{R}}}\}$ is a resolution class of $\mathcal{R}$. Therefore, for every block $B \in \mathcal{B}_m$ containing the pair $(c_{m_1},A_{j_1})(c_{m_2},A_{j_2})$, the pair $(c_{m_1}, a_{j_1}^{q_1})(c_{m_2}, a_{j_2}^{q_2})$ appears exactly once in the block $B^\prime \in \mathcal{B}_m^\prime$. Furthermore, for any block $B\in \mathcal{B}_m$ not containing the pair $(c_{m_1},A_{j_1})(c_{m_2},A_{j_2})$, the block $B^\prime \in \mathcal{B}_m^\prime$ does not contain the pair $(c_{m_1},a_{j_1}^{q_1}) (c_{m_2},a_{j_2}^{q_2})$. So the pair $(c_{m_1},a_{j_1}^{q_1}) (c_{m_2},a_{j_2}^{q_2})$ appears in exactly $\frac{\lambda}{\ell - 2}$ blocks in $\mathcal{B}_m^\prime$. Consequently, the pair $(c_{m_1},a_{j_1}^{q_1}) (c_{m_2},a_{j_2}^{q_2})$ appears in exactly $(\ell-2) \frac{\lambda}{(\ell-2)} = \lambda$ blocks in $\mathcal{B}^\prime$. 

    Suppose instead that $m_1 = m_2 = M$ for some $0\leq M\leq \ell-1$. Then, for $(c_{M},a_{j_1}^{q_1})$ and $(c_{M},a_{j_2}^{q_2})$ to be distinct, we must have $a_{j_1}^{q_1} \neq a_{j_2}^{q_2}$. For each $0\leq m\leq \ell-1$, $\mathcal{T}_m$ contains blocks of the form \[B = \{(c_{m+1}, A_{j_0}), (c_{m+2}, A_{j_1}), \dots, (c_{m+\ell-1}, A_{j_{\ell-2}})\}.\] Thus, there is no point in $\mathcal{T}_M$ that has $c_M$ as its first coordinate, and so the pair $(c_{M},a_{j_1}^{q_1})(c_{M},a_{j_2}^{q_2})$ does not appear in any block of $\mathcal{T}_M$. Fix $0\leq m\leq \ell - 1$ such that $m\neq M$. If $B$ is a block in $\mathcal{T}_m$, then $B$ contains exactly one of $(c_M, A_1)$, $(c_M, A_2)$, \dots, $(c_M, A_{b_\mathcal{R}})$. Thus $B^\prime$ contains exactly one of the following subsets:
    \begin{align*}
        &\left\{(c_M,a_1^1), \dots, (c_M,a_1^{\frac{k}{\ell-2}})\right\},\\
        &\left\{(c_M,a_2^{\frac{k}{\ell-2} + 1}), \dots, (c_M,a_2^{(2)\frac{k}{\ell-2}})\right\},\\
        &\vdots\\
        &\left\{(c_M,a_{b_{\mathcal{R}}}^{(s_{\mathcal{R}} - 1)\frac{k}{\ell-2} + 1}), \dots, (c_M,a_{b_{\mathcal{R}}}^{\frac{v}{\ell}})\right\}.
    \end{align*}

    Recall that the pair $a_{j_1}^{q_1} a_{j_2}^{q_2}$ appears in exactly $\frac{k\ell(\ell-2)}{v(\ell-1)}$ blocks of $\mathcal{R}$. So, out of the above subsets, there are $\frac{k\ell(\ell-2)}{v(\ell-1)}$ subsets that contain the pair $(c_M, a_{j_1}^{q_1}) (c_M, a_{j_2}^{q_2})$. Furthermore, for any point $(c_M, A_j)$ in $\mathcal{T}_m$ such that $a_{j_1}^{q_1}, a_{j_2}^{q_2}\in A_j$, $(c_M, A_j)$ occurs in exactly $\frac{\lambda v (\ell-1)}{k\ell(\ell-2)}$ blocks by Lemma~\ref{lem: r and b for TDs}. Thus, the pair $(c_{M},a_{j_1}^{q_1}) (c_{M},a_{j_2}^{q_2})$ is contained in exactly \[\left(\frac{k\ell(\ell-2)}{v(\ell-1)}\right)\left( \frac{\lambda v (\ell-1)}{k\ell(\ell-2)} \right) = \lambda\] blocks. 

    Therefore, we have constructed a $(v,k,\lambda)$-BIBD. To see that this BIBD admits a $0$-ULSE $\ell$-colouring, treat $\{c_0, \dots, c_{\ell-1}\}$ as a set of colours. For each point of the form $(c_i,a_j^q) \in X$, assign it the colour $c_i$. Then every block will contain exactly $\ell-1$ colours with each of these colours being represented exactly $\frac{k}{\ell-1}$ times. 
\end{proof}

To further illustrate how the proof of Theorem~\ref{thm: general construction for nice 0-LSE l-coloured BIBDs} works, we now give an example of how to construct a $0$-ULSE $5$-coloured $(45,12,3)$-BIBD, using the notation from the theorem's proof. Let $P = \{c_0, c_1, c_2, c_3, c_4\}$ and $E_m = P\backslash \{c_m\}$ for each $m\in \{0,1,2,3,4\}$. Let $\mathcal{R}$ be the $(9,3,1)$-RBIBD shown in Table~\ref{tab: (9,3,1)-RBIBD}.

\begin{table}[ht]
    \centering
        \begin{tabular}{c|c|c|c}
            $\Pi_1$ & $\Pi_2$ & $\Pi_3$ & $\Pi_4$ \\
            \hline
            $A_1 = \{0,1,2\}$ & $A_4 = \{0,3,6\}$ & $A_7 = \{0,4,8\}$ & $A_{10} = \{0,5,7\}$\\
            $A_2 = \{3,4,5\}$ & $A_5 = \{1,4,7\}$ & $A_8 = \{1,5,6\}$ & $A_{11} = \{1,3,8\}$\\
            $A_3 = \{6,7,8\}$ & $A_6 = \{2,5,8\}$ & $A_9 = \{2,3,7\}$ & $A_{12} = \{2,4,6\}$
        \end{tabular}
    \caption{A $(9,3,1)$-RBIBD on the point set $V = \{0,1,2,3,4,5,6,7,8\}$.}
    \label{tab: (9,3,1)-RBIBD}
\end{table}

For each $m \in \{0,1,2,3,4\}$, we construct the TD$_1(4,3)$ $\mathcal{T}_m$ using the following set of points:
\begin{align*}
    \{& (c_{m+1}, A_1), (c_{m+1}, A_2), (c_{m+1}, A_3), \\
    &(c_{m+2}, A_4), (c_{m+2}, A_5), (c_{m+2}, A_6), \\
    &(c_{m+3}, A_7), (c_{m+3}, A_8), (c_{m+3}, A_9), \\
    &(c_{m+4}, A_{10}), (c_{m+4}, A_{11}), (c_{m+4}, A_{12})\}
\end{align*}
where addition in the subscripts is done modulo $5$. The blocks $\mathcal{B}_m$ of $\mathcal{T}_m$ are illustrated in Figure~\ref{fig: blocks of T_m}.

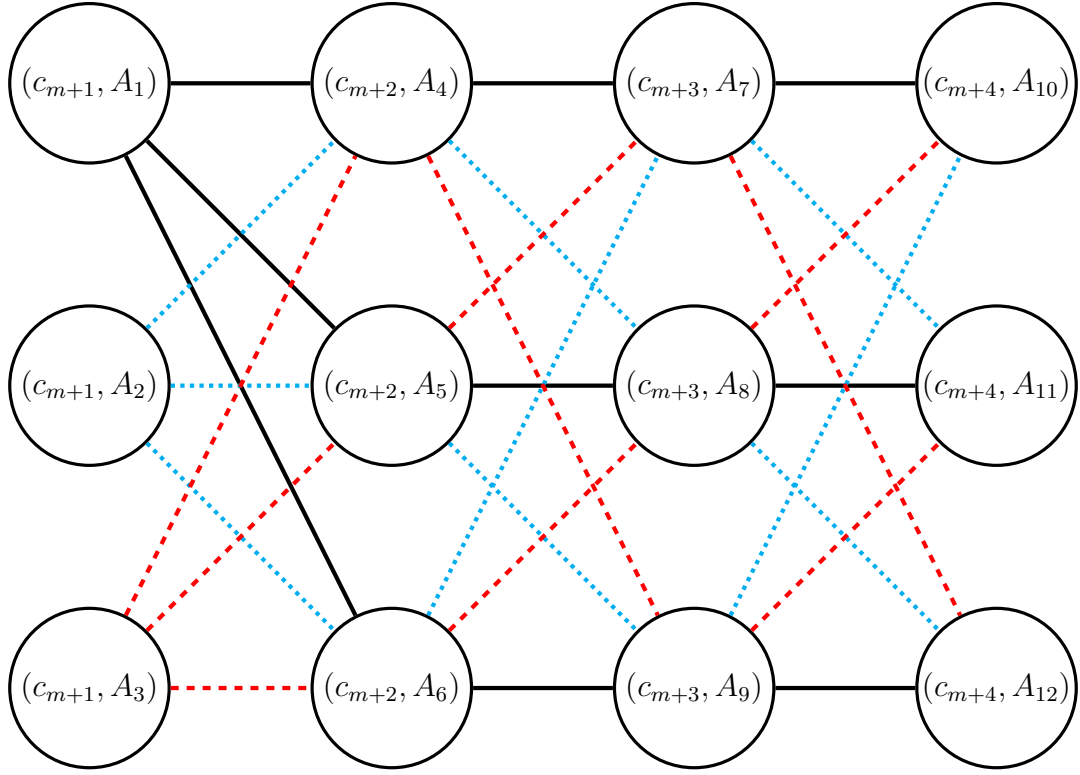
\begin{figure}
    \centering
    \begin{tikzpicture}[scale = 2]
        \tikzstyle{vertex}=[circle, draw=black, fill=white, very thick, minimum size=60pt,inner sep=0pt]
            \node[vertex] at (0,0) (0) {$(c_{m+1}, A_1)$};
            \node[vertex] at (0,-2) (1) {$(c_{m+1}, A_2)$};
            \node[vertex] at (0,-4) (2) {$(c_{m+1}, A_3)$};
            \node[vertex] at (2,0) (3) {$(c_{m+2}, A_4)$};
            \node[vertex] at (2,-2) (4) {$(c_{m+2}, A_5)$};
            \node[vertex] at (2,-4) (5) {$(c_{m+2}, A_6)$};
            \node[vertex] at (4,0) (6) {$(c_{m+3}, A_7)$};
            \node[vertex] at (4,-2) (7) {$(c_{m+3}, A_8)$};
            \node[vertex] at (4,-4) (8) {$(c_{m+3}, A_9)$};
            \node[vertex] at (6,0) (9) {$(c_{m+4}, A_{10})$};
            \node[vertex] at (6,-2) (10) {$(c_{m+4}, A_{11})$};
            \node[vertex] at (6,-4) (11) {$(c_{m+4}, A_{12})$};

            \draw[ultra thick] (0)--(3)--(6)--(9);
            \draw[ultra thick] (0)--(4)--(7)--(10);
            \draw[ultra thick] (0)--(5)--(8)--(11);

            \draw[ultra thick, cyan, dotted] (1)--(3)--(7)--(11);
            \draw[ultra thick, cyan, dotted] (1)--(4)--(8)--(9);
            \draw[ultra thick, cyan, dotted] (1)--(5)--(6)--(10);

            \draw[ultra thick, red, dashed] (2)--(3)--(8)--(10);
            \draw[ultra thick, red, dashed] (2)--(4)--(6)--(11);
            \draw[ultra thick, red, dashed] (2)--(5)--(7)--(9);
        \end{tikzpicture}
    \caption{The set of blocks $\mathcal{B}_m$ for the TD$_1(4,3)$, $\mathcal{T}_m$.}
    \label{fig: blocks of T_m}
\end{figure}

We now construct $\mathcal{B}^\prime_m$ as follows. Let $B \in \mathcal{B}_m$; for example, \[B = \{(c_{m+1},A_1), (c_{m+2},A_{4}), (c_{m+3},A_{7}), (c_{m+4},A_{10})\}.\] We let $B^\prime = \{(c_{m+i}, a) \mid 1\leq i \leq 4, a\in A_1 \cup A_4 \cup A_7 \cup A_{10}, (c_{m+i}, A_j) \in B\}$. Listing out all of the elements yields
\begin{align*}
    B^\prime = \{&(c_{m+1},0), (c_{m+1},1), (c_{m+1},2), \\ 
    &(c_{m+2},0),(c_{m+2},3), (c_{m+2},6), \\
    &(c_{m+3},0), (c_{m+3},4), (c_{m+3},8), \\
    &(c_{m+4},0), (c_{m+4},5), (c_{m+4},7)\}.
\end{align*}
We repeat this transformation for every block in $\mathcal{B}_m$ and the set of all of these transformed blocks gives us $\mathcal{B}^\prime_m$. This concludes the construction of the $0$-ULSE $5$-coloured $(45,12,3)$-BIBD.

For the rest of this section, we work towards an alternate version of Theorem~\ref{thm: general construction for nice 0-LSE l-coloured BIBDs} that establishes a way of constructing symmetric $0$-ULSE coloured BIBDs.

Applying Theorem~\ref{thm: symmetric 0-ULSE lambda value} and substituting Equation~\eqref{eq: k=0 mod l-1, soln for v} into the parameters for the input ingredient RBIBD in Theorem~\ref{thm: general construction for nice 0-LSE l-coloured BIBDs} yields the parameters \[\left(\frac{(\ell-2)k}{(\ell-1)^2 - k}, \frac{k}{\ell-1}, \frac{(\ell-1)^2 - k}{\ell-1}\right).\] Observe that a $\left(\frac{(\ell-2)k}{(\ell-1)^2 - k}, \frac{k}{\ell-1}, \frac{(\ell-1)^2 - k}{\ell-1}\right)$-RBIBD satisfies the property that its replication number is equal to the sum of its block size and its index. Any RBIBD that admits this property is called \emph{affine resolvable}. The following lemma concerning the intersection of blocks in affine resolvable BIBDs was proven by Bose \cite{B42}.

\begin{lemma}\label{lem: affine resolvable intersection of blocks}
    Let $\mathcal{D}$ be an affine resolvable $(v,k,\lambda)$-BIBD. If $B_1$ and $B_2$ are two blocks of $\mathcal{D}$ in different resolution classes then $|B_1 \cap B_2| = \frac{k^2}{v}$.
\end{lemma}

\begin{lemma}\label{lem: ingredient RBIBD is an affine resolvable BIBD}
    If $k,\ell\in \mathbb{Z}^+$ and $\ell\geq 4$, then an $\left(\frac{(\ell-2)k}{(\ell-1)^2 - k}, \frac{k}{\ell-1}, \frac{(\ell-1)^2 - k}{\ell-1}\right)$-RBIBD is affine resolvable. Furthermore, for any two blocks $B_i$ and $B_j$ in different resolution classes, \[|B_i \cap B_j| = \frac{k((\ell-1)^2 - k)}{(\ell-1)^2(\ell-2)}.\]
\end{lemma}

\begin{proof}
    The replication number of the RBIBD is \[\ell-1 = \frac{k}{\ell-1} + \frac{(\ell-1)^2 - k}{\ell-1}\] and so it is affine resolvable. By Lemma~\ref{lem: affine resolvable intersection of blocks}, if $B_i$ and $B_j$ are in different resolution classes, \[|B_i \cap B_j| = \frac{\left(\frac{k}{\ell-1}\right)^2}{\left(\frac{(\ell-2)k}{(\ell-1)^2 - k}\right)} = \frac{k((\ell-1)^2 - k)}{(\ell-1)^2(\ell-2)}. \qedhere\]
\end{proof}

If we use Theorem~\ref{thm: general construction for nice 0-LSE l-coloured BIBDs} to construct a symmetric $0$-LSE $\ell$-coloured BIBD, then the transversal design from Theorem~\ref{thm: general construction for nice 0-LSE l-coloured BIBDs} is a \[\text{TD}_{\frac{k((\ell-1)^2 - k)}{(\ell-1)^2(\ell-2)}}\left(\ell-1, \frac{(\ell-1)(\ell-2)}{(\ell-1)^2 - k}\right).\] We now prove that the RBIBD and TD from Theorem~\ref{thm: general construction for nice 0-LSE l-coloured BIBDs} are duals of each other in the case where the resulting $0$-ULSE $\ell$-coloured BIBD in Theorem~\ref{thm: general construction for nice 0-LSE l-coloured BIBDs} is symmetric.

\begin{lemma}\label{lem: ingredients are duals in symmetric case}
    There exists an $\left(\frac{(\ell-2)k}{(\ell-1)^2 - k}, \frac{k}{\ell-1}, \frac{(\ell-1)^2 - k}{\ell-1}\right)$-RBIBD if and only if there exists a TD$_{\frac{k((\ell-1)^2 - k)}{(\ell-1)^2(\ell-2)}}\left(\ell-1, \frac{(\ell-1)(\ell-2)}{(\ell-1)^2 - k}\right)$. 
\end{lemma}

\begin{proof}
    Let $\mathcal{R}$ be an $\left(\frac{(\ell-2)k}{(\ell-1)^2 - k}, \frac{k}{\ell-1}, \frac{(\ell-1)^2 - k}{\ell-1}\right)$-RBIBD. Let $x_1, \dots, x_{\frac{(\ell-2)k}{(\ell-1)^2 - k}}$ be the points of $\mathcal{R}$ and let $B_1, \dots, B_{\frac{(\ell-1)^2(\ell-2)}{(\ell-1)^2 - k}}$ be the blocks of $\mathcal{R}$. We construct the dual design $\mathcal{T}$ on the points $y_1, \dots, y_{\frac{(\ell-1)^2(\ell-2)}{(\ell-1)^2 - k}}$ and with blocks $D_1, \dots, D_{\frac{(\ell-2)k}{(\ell-1)^2 - k}}$ where $y_i \in D_j$ if and only if $x_j \in B_i$. We claim that $\mathcal{T}$ is a TD$_{\frac{k((\ell-1)^2 - k)}{(\ell-1)^2(\ell-2)}}\left(\ell-1, \frac{(\ell-1)(\ell-2)}{(\ell-1)^2 - k}\right)$. 

    Since $\mathcal{R}$ has $\frac{(\ell-1)^2(\ell-2)}{(\ell-1)^2 - k}$ blocks, $\mathcal{T}$ has $\frac{(\ell-1)^2(\ell-2)}{(\ell-1)^2 - k}$ points. Since $y_i \in D_j$ if and only if $x_j \in B_i$, $|D_j|$ is equal to the number of blocks containing $x_j$ in $\mathcal{R}$. The replication number of $\mathcal{R}$ is 
    \begin{align*}
        \frac{(\ell-1)^2 - k}{\ell-1} \left( \frac{(\ell-2)k}{(\ell-1)^2 - k} - 1 \right)\left( \frac{1}{\frac{k}{\ell-1} - 1} \right) &= \frac{\ell k - k - \ell^2 + 2\ell - 1}{k - \ell + 1} \\
        &= \frac{(\ell-1)(k-\ell+1)}{k-\ell+1} \\
        &=\ell-1.
    \end{align*}
    Thus the block size of $\mathcal{T}$ is also $\ell - 1$. 

    Consider the $\ell-1$ resolution classes of the RBIBD. There are $\frac{(\ell-1)(\ell-2)}{(\ell-1)^2 - k}$ blocks in each resolution class. The blocks in each resolution class are pairwise disjoint and the union of all of the blocks in a resolution class is the entire point set of the RBIBD. Without loss of generality, let $\left\{B_1, \dots, B_{\frac{(\ell-1)(\ell-2)}{(\ell-1)^2 - k}}\right\}$ be a resolution class. Translating into $\mathcal{T}$, we observe that no block will contain any two of the points $y_1, \dots, y_{\frac{(\ell-1)(\ell-2)}{(\ell-1)^2 - k}}$ and every block will contain at least one of the points $y_1, \dots, y_{\frac{(\ell-1)(\ell-2)}{(\ell-1)^2 - k}}$. Therefore the points $y_1, \dots, y_{\frac{(\ell-1)(\ell-2)}{(\ell-1)^2 - k}}$ form a group. Since there are $\ell-1$ resolution classes in the RBIBD, there are $\ell-1$ groups in $\mathcal{T}$, as we desired.

    All that remains to show is that every pair of points that do not share a group share exactly $\frac{k((\ell-1)^2 - k)}{(\ell-1)^2(\ell-2)}$ blocks. Let $y_i$ and $y_j$ be two points from different groups in $\mathcal{T}$. Then the pair $y_i y_j$ is contained in exactly $\frac{k((\ell-1)^2 - k)}{(\ell-1)^2(\ell-2)}$ blocks if and only if $|B_i \cap B_j| = \frac{k((\ell-1)^2 - k)}{(\ell-1)^2(\ell-2)}$ where $B_i$ and $B_j$ are blocks in different resolution classes. We know this to be true by Lemma~\ref{lem: ingredient RBIBD is an affine resolvable BIBD} and so $\mathcal{T}$ is a TD$_{\frac{k((\ell-1)^2 - k)}{(\ell-1)^2(\ell-2)}}(\ell-1, \frac{(\ell-1)(\ell-2)}{(\ell-1)^2 - k})$. 

    Since the dual of the dual of a design is the original design, the dual of $\mathcal{T}$ is $\mathcal{R}$. This proves the result.
\end{proof}

We now state an alternate version of Theorem~\ref{thm: general construction for nice 0-LSE l-coloured BIBDs} where only the existence of the RBIBD ingredient design is needed for constructing a symmetric $0$-ULSE $\ell$-coloured BIBD.

\begin{theorem}
    If there exists an $\left(\frac{(\ell-2)k}{(\ell-1)^2 - k}, \frac{k}{\ell-1}, \frac{(\ell-1)^2 - k}{\ell-1}\right)$-RBIBD and $\lambda$ is a multiple of $\frac{k((\ell-1)^2 - k)}{(\ell-1)^2}$, then there exists a $0$-ULSE $\ell$-coloured $(v,k,\lambda)$-BIBD. In particular, if there exists an $\left(\frac{(\ell-2)k}{(\ell-1)^2 - k}, \frac{k}{\ell-1}, \frac{(\ell-1)^2 - k}{\ell-1}\right)$-RBIBD, then there exists a symmetric $0$-ULSE $\ell$-coloured $(v,k,\lambda)$-BIBD.
\end{theorem}

\begin{proof}
    From Lemma~\ref{lem: ingredients are duals in symmetric case}, the existence of a $\left(\frac{(\ell-2)k}{(\ell-1)^2 - k}, \frac{k}{\ell-1}, \frac{(\ell-1)^2 - k}{\ell-1}\right)$-RBIBD implies the existence of a TD$_{\frac{k((\ell-1)^2 - k)}{(\ell-1)^2(\ell-2)}}\left(\ell-1, \frac{(\ell-1)(\ell-2)}{(\ell-1)^2 - k}\right)$. From Theorem~\ref{thm: general construction for nice 0-LSE l-coloured BIBDs}, we obtain a $0$-ULSE $\ell$-coloured  \[\left(\frac{\ell(\ell-2)k}{(\ell-1)^2 - k}, k, \frac{k((\ell-1)^2 - k)}{(\ell-1)^2}\right)\!\text{-BIBD.}\] Recall from Equation~\eqref{eq: k=0 mod l-1, soln for v} that $v = \frac{\ell(\ell-2)k}{(\ell-1)^2 - k}$. It can be easily verified that the above is a symmetric BIBD. If $\lambda = m\frac{k((\ell-1)^2 - k)}{(\ell-1)^2}$ for some $m\in \mathbb{Z}^+$, then a $0$-ULSE $\ell$-coloured $(v,k,\lambda)$-BIBD can be obtained by duplicating each block of the above $\left(\frac{\ell(\ell-2)k}{(\ell-1)^2 - k}, k, \frac{k((\ell-1)^2 - k)}{(\ell-1)^2}\right)\!\text{-BIBD}$ $m$ times, similar to the proof of Theorem~\ref{thm: duplicating blocks}. 
\end{proof}

\section{Applications of Theorem~\ref{thm: general construction for nice 0-LSE l-coloured BIBDs}}\label{sec: Applications of general construction}

Theorem~\ref{thm: general construction for nice 0-LSE l-coloured BIBDs} states that given a specific transversal design and resolvable BIBD, we can construct a $0$-ULSE coloured BIBD. We now discuss two ways of furnishing these ingredient designs: Hadamard matrices and affine planes.

\subsection{Hadamard Matrices}\label{sec: Hadamard matrices}

A \emph{Hadamard matrix} of order $n\in \mathbb{Z}^+$, denoted $H(n)$, is an $n\times n$ matrix with entries in $\{-1, 1\}$ such that the rows are pairwise orthogonal. It is commonly known that if a Hadamard matrix of order $n$ exists, then $n\in \{1,2\}$ or $n\equiv 0 \pmod{4}$. Hadamard \cite{H1893} conjectured that this necessary condition is also sufficient and the conjecture is still open. See Part V of \cite{HCD} for more on Hadamard matrices. A Hadamard matrix $H(4n)$ exists if and only if a $(4n-1, 2n-1, n-1)$-BIBD exists; a proof of this fact can be found in Chapter 4 of \cite{S04}.

A Hadamard matrix is \emph{regular} if the sum of the entries in each row and the sum of the entries in each column of the matrix are the same. A \emph{Menon design} is a BIBD with parameters either $(4u^2, 2u^2 + u, u^2 + u)$ or $(4u^2, 2u^2 - u, u^2 - u)$. It is well-known that there exists a regular $H(n)$ if and only if a Menon design on $n$ points exists (see Chapter 4 of \cite{S04}). Menon designs of the form $(4u^2, 2u^2 - u, u^2 - u)$ satisfy Equation~\eqref{eq: k=0 mod l-1, soln for v} when substituting $\ell = 2u$. This substitution yields a design with parameters $\left(\ell^2, \binom{\ell}{2}, \frac{\ell(\ell - 2)}{4}\right)$. 

Using Theorem~\ref{thm: general construction for nice 0-LSE l-coloured BIBDs} and the following lemmas, we can construct $0$-ULSE $\ell$-coloured $\left(\ell^2, \binom{\ell}{2}, \frac{\ell(\ell - 2)}{4}\right)$-BIBDs using Hadamard matrices of order $\ell$.

\begin{lemma}\label{lem: H(4n) = TD_ell/4(ell - 1, 2)}
    For $\ell \equiv 0 \pmod{4}$, a Hadamard matrix $H(\ell)$ exists if and only if a TD$_{\frac{\ell}{4}}(\ell - 1, 2)$ exists.
\end{lemma}

\begin{lemma}\label{lem: H(4n) = (4n, 4n/2, 4n/2 - 1)-RBIBD}
    For $\ell \equiv 0 \pmod{4}$, a Hadamard matrix $H(\ell)$ exists if and only if an $\left(\ell, \frac{\ell}{2}, \frac{\ell}{2} - 1\right)$-RBIBD exists.
\end{lemma}

Lemma~\ref{lem: H(4n) = TD_ell/4(ell - 1, 2)} can be obtained by combining results from \cite{H74, T33}. A proof for Lemma~\ref{lem: H(4n) = (4n, 4n/2, 4n/2 - 1)-RBIBD} can be found in \cite{BS70}.

\begin{theorem}\label{thm: Hadamard matrix gives 0-ULSE coloured BIBD}
    Let $\ell \equiv 0 \pmod{4}$. If there exists a Hadamard matrix $H(\ell)$ then there exists a $\left(\ell^2, \binom{\ell}{2}, \frac{\ell(\ell - 2)}{4}\right)$-BIBD with a $0$-ULSE $\ell$-colouring. 
\end{theorem}

\begin{proof}
    The result follows directly from Lemma~\ref{lem: H(4n) = TD_ell/4(ell - 1, 2)}, Lemma~\ref{lem: H(4n) = (4n, 4n/2, 4n/2 - 1)-RBIBD}, and Theorem~\ref{thm: general construction for nice 0-LSE l-coloured BIBDs}.
\end{proof}

\subsection{Affine Planes}

An \emph{affine plane} of order $n$ is an $(n^2, n, 1)$-BIBD. Using well known results on affine planes, we can obtain the ingredients for Theorem~\ref{thm: general construction for nice 0-LSE l-coloured BIBDs}.

\begin{lemma}[See Theorems 6.32 and 6.44 in \cite{S04}]\label{lem: affine plane implies TD}
    There exists an affine plane of order $n$ if and only if there exists a TD$(n+1,n)$.
\end{lemma}

A \emph{projective plane} of order $n$ is an $(n^2 + n + 1, n+1, 1)$-BIBD. It is well known that a projective plane of order $n$ exists if and only if an affine plane of order $n$ exists; see \cite{S04} for a proof. It was originally proven in \cite{VB1906} that there exists a projective plane of order $n$ when $n$ is a power of a prime. Consequently, we have the following existence result for affine planes.

\begin{lemma}\label{lem: affine planes of prime power order}
    If $n\geq 2$ is a prime power, then there exists an affine plane of order $n$.
\end{lemma}

Determining if there exists a projective plane of an order that is not a power of a prime is a long-standing open problem. Proving that there is no projective plane of order $6$ is credited to Tarry \cite{T1900} and proving that there is no projective plane of order $10$ was done by Lam, Thiel, and Swiercz \cite{LTS89}. It remains open whether there exists a projective plane of order $12$, which is the smallest open case. 

Lemma~\ref{lem: affine plane implies TD} shows how we obtain one of the ingredient designs in Theorem~\ref{thm: general construction for nice 0-LSE l-coloured BIBDs} from an affine plane. The following known property of affine planes shows that affine planes can be used as the second ingredient. 

\begin{lemma}[See Theorem 5.9 in \cite{S04}]\label{lem: affine planes are resolvable}
    If there exists an affine plane of order $n$, then it is resolvable. 
\end{lemma}

\begin{theorem}\label{thm: affine plane gives 0-LSE l-col BIBD}
    If $\ell \in \mathbb{Z}^+$ such that there exists an affine plane of order $\ell-2$, then there exists a $(\ell(\ell-2)^2, (\ell-1)(\ell-2), \ell-2)$-BIBD with a $0$-ULSE $\ell$-colouring.
\end{theorem}

\begin{proof}
    The result follows by applying Lemmas~\ref{lem: affine plane implies TD}, \ref{lem: affine planes of prime power order}, and \ref{lem: affine planes are resolvable} to Theorem~\ref{thm: general construction for nice 0-LSE l-coloured BIBDs}.
\end{proof}

In \cite{S69}, it is proven that if a Hadamard matrix of order $4n$ exists, then a $((4n)^2, 2n(4n-1), 2n(2n-1))$-BIBD exists. It was shown in \cite{AS69} that given a prime power $q$, a $((q+2)q^2, (q+1)q, q)$-BIBD can be constructed. Theorems~\ref{thm: Hadamard matrix gives 0-ULSE coloured BIBD} and \ref{thm: affine plane gives 0-LSE l-col BIBD} strengthen these results by adding a colouring condition to the output BIBDs in the results from \cite{AS69, S69}.

\section{Twin Prime Powers}\label{sec: TPP Construction}

Not every $0$-ULSE $\ell$-coloured BIBD is captured by Theorem~\ref{thm: general construction for nice 0-LSE l-coloured BIBDs}. In this section, we construct an infinite family that cannot be constructed using Theorem~\ref{thm: general construction for nice 0-LSE l-coloured BIBDs}. To do this, we start by stating an elementary result concerning finite fields.

\begin{lemma}\label{lem: finite fields result}
    If $\mathbb{F}_q$ is a finite field of order $q = p^n$ where $p$ is an odd prime and $n\in \mathbb{Z}^+$, then there are $\frac{q-1}{2}$ nonzero squares in $\mathbb{F}_q$ and $\frac{q-1}{2}$ nonsquares. 
\end{lemma}

A positive integer $n$ is called a \emph{prime power} if $n = p^\omega$ where $p$ is a prime and $\omega \in \mathbb{Z}^+$. Two integers $n_1$ and $n_2$ are called \emph{twin prime powers} if $n_1$ and $n_2$ are prime powers, and $|n_1 - n_2| = 2$.

Let $G$ be a finite group of order $v$ with identity element $0$. A $(v,k,\lambda)$ \emph{difference set} $D\subseteq G$ is a set of $k$ group elements such that the multiset $\{x-y \mid x,y\in D, x\neq y\}$ is the multiset of all elements of $G\backslash \{0\}$ repeated $\lambda$ times. Difference sets can be used to construct symmetric BIBDs (see Chapter 3 of \cite{S04}). Given twin prime powers, $\ell$ and $\ell - 2$, Stanton and Sprott in \cite{SS58} prove the existence of an $\left(\ell(\ell - 2), \frac{\ell^2 - 1}{2} - \ell, \frac{(\ell - 3)(\ell +1)}{4}\right)$ difference set.

\begin{lemma}[\cite{SS58}]\label{lem: twin prime power difference set}
    Let $\ell$ and $\ell - 2$ be prime powers. The \emph{twin prime power difference set} is defined to be \[\{ (x,y) \in \mathbb{F}_{\ell -2}\times \mathbb{F}_{\ell} \mid \text{$x$ and $y$ are both nonzero squares, or both nonsquares, or $y=0$} \}\] where the differences are calculated over the additive group of the field. This set generates an $\left(\ell(\ell - 2), \frac{\ell^2 - 1}{2} - \ell, \frac{(\ell - 3)(\ell +1)}{4}\right)$-BIBD.
\end{lemma}

\begin{lemma}\label{lem: l-2 SE l coloured twin prime power difference BIBDs}
    Let $\ell$ and $\ell - 2$ be twin prime powers. The $\left(\ell(\ell - 2), \frac{\ell^2 - 1}{2} - \ell, \frac{(\ell - 3)(\ell +1)}{4}\right)$-BIBDs with $\ell \geq 5$ generated by twin prime power difference sets are $(\ell-2)$-ULSE $\ell$-colourable.
\end{lemma}

\begin{proof}

    For each $y \in \mathbb{F}_\ell$, let $C_y = \{(x,y) \mid x\in \mathbb{F}_{\ell-2}\}$ be a colour class. For each $(z_1, z_2) \in \mathbb{F}_{\ell-2} \times \mathbb{F}_\ell$, let
    \[B_{(z_1, z_2)} = \{(x+z_1, y+z_2) \mid \text{$(x,y)$ is in the twin prime power difference set}\}.\] The set of blocks $\{B_{(z_1, z_2)} \mid (z_1, z_2) \in \mathbb{F}_{\ell-2} \times \mathbb{F}_\ell\}$ form an $\left(\ell(\ell - 2), \frac{\ell^2 - 1}{2} - \ell, \frac{(\ell - 3)(\ell +1)}{4}\right)$-BIBD.

    Fix $(z_1, z_2) \in \mathbb{F}_{\ell-2} \times \mathbb{F}_\ell$. Then $C_{z_2} \subseteq B_{(z_1, z_2)}$. Let $m\in \mathbb{F}_\ell \backslash \{z_2\}$. Either $m - z_2 \in \mathbb{F}_{\ell}$ is a nonzero square or a nonsquare. Suppose $m-z_2$ is a nonzero square. By Lemma~\ref{lem: finite fields result} there are $\frac{\ell-3}{2}$ nonzero squares in $\mathbb{F}_{\ell-2}$. Therefore there are $\frac{\ell-3}{2}$ elements $x\in \mathbb{F}_{\ell-2}$ such that $(x, m-z_2)$ is in the twin prime power difference set. So $|C_m \cap B_{(z_1, z_2)}| = \frac{\ell-3}{2}$. By a nearly identical argument, if $m - z_2$ is a nonsquare, then $|C_m \cap B{(z_1, z_2)}| = \frac{\ell-3}{2}$. 

    Thus, within every block of the $\left(\ell(\ell - 2), \frac{\ell^2 - 1}{2} - \ell, \frac{(\ell - 3)(\ell +1)}{4}\right)$-BIBD, one colour appears $\ell-2$ times and the rest appear exactly $\frac{\ell-3}{2}$ times. This forms an $(\ell-2)$-ULSE $\ell$-colouring.
\end{proof}

By taking the complement of the BIBDs from Lemma~\ref{lem: l-2 SE l coloured twin prime power difference BIBDs}, we obtain a family of BIBDs that admit $0$-ULSE colourings.

\begin{theorem}\label{thm: twin prime powers give 0-ULSE BIBD}
    If $\ell\in \mathbb{Z}$ such that $\ell \geq 5$, $\ell$ is a prime power, and $\ell - 2$ is a prime power, then there exists a $0$-ULSE $\ell$-coloured $\left(\ell(\ell - 2), \frac{(\ell - 1)^2}{2}, \frac{(\ell - 1)^2}{4}\right)$-BIBD.
\end{theorem}

\begin{proof}
    Applying Corollary~\ref{cor: complements of t-LSE colourings} to Lemma~\ref{lem: l-2 SE l coloured twin prime power difference BIBDs} immediately gives the result. 
\end{proof}

Regarding the $\left(\ell(\ell - 2), \frac{(\ell - 1)^2}{2}, \frac{(\ell - 1)^2}{4}\right)$-BIBDs from Theorem~\ref{thm: twin prime powers give 0-ULSE BIBD}, if we substitute $4n-1 = \ell(\ell-2)$ into the parameters of the BIBDs we obtain $(4n-1, 2n, n)$. The complement of a $(4n-1, 2n, n)$-BIBD is a $(4n-1, 2n-1, n-1)$-BIBD which, recalling from Section~\ref{sec: Hadamard matrices}, exists if and only if a Hadamard matrix of order $4n$ exists. Pairing this observation with Theorem~\ref{thm: Hadamard matrix gives 0-ULSE coloured BIBD} establishes that $0$-ULSE coloured BIBDs can be constructed from Hadamard matrices, and Hadamard matrices can be constructed from $0$-ULSE coloured BIBDs. It is already known that the twin prime power difference set can be used to construct Hadamard matrices (see for example \cite{OCS10}).

To see that the $0$-LSE $\ell$-coloured $\left(\ell(\ell-2), \frac{(\ell-2)^2}{2}, \frac{(\ell-1)^4}{4}\right)$-BIBDs from Theorem~\ref{thm: twin prime powers give 0-ULSE BIBD} cannot be constructed from Theorem~\ref{thm: general construction for nice 0-LSE l-coloured BIBDs}, consider the ingredients if the construction were possible. Note that $\frac{(\ell-1)^2}{4(\ell-2)}\notin \mathbb{Z}$ since gcd$(\ell-1, \ell-2) = 1$ and $\ell \geq 4$. However, the construction from Theorem~\ref{thm: general construction for nice 0-LSE l-coloured BIBDs} would a require a transversal design of index $\frac{(\ell-1)^2}{4(\ell-2)}$ which is not possible.

As a final note regarding all three infinite families of $0$-ULSE $\ell$-coloured BIBDs from Theorems~\ref{thm: Hadamard matrix gives 0-ULSE coloured BIBD}, \ref{thm: affine plane gives 0-LSE l-col BIBD}, and \ref{thm: twin prime powers give 0-ULSE BIBD}, it can be easily verified that the only BIBD that is contained in more than one family is the $0$-ULSE $4$-coloured $(16,6,2)$-BIBD built by Hadamard matrices and affine planes.

\section{Open Problems}\label{sec: Open Problems}

From Theorems \ref{thm: Hadamard matrix gives 0-ULSE coloured BIBD} and \ref{thm: affine plane gives 0-LSE l-col BIBD}, we have a $0$-ULSE $\ell$-coloured BIBD for every $\ell \equiv 0 \pmod{4}$ where a Hadamard matrix of order $\ell$ exists and also for every $\ell \geq 5$ where $\ell - 2$ is a prime power. If we assume Hadamard's conjecture is true, then it remains open whether there exists a $0$-ULSE $\ell$-coloured BIBD for the following cases: 
\begin{itemize}
    \item[(i)] $\ell \equiv 2 \pmod{4}$ and $\ell \geq 6$, and
    \item[(ii)] $\ell \geq 5$ is odd and $\ell-2$ is not a prime power.
\end{itemize}

Table~\ref{tab: k=0 mod l-1, possible BIBDs} gives a list of all values for $v$, $k$, $\ell$, and $\lambda$ that satisfy Equation~\eqref{eq: k=0 mod l-1, soln for v} and Lemmas~\ref{lem: replication num of a BIBD} and \ref{lem: num of blocks in a BIBD} for $4\leq \ell \leq 14$. The value for the index is the $\lambda_{\min}$ from Theorem~\ref{thm: symmetric 0-ULSE lambda value}. For the families of BIBDs found in this paper, we refer to their corresponding theorems in the rightmost column of Table~\ref{tab: k=0 mod l-1, possible BIBDs}. An entry in the rightmost column of Table~\ref{tab: k=0 mod l-1, possible BIBDs} is left blank if there is no known $0$-ULSE coloured BIBD with the parameters in the corresponding row.

\begin{table}[p]
    \centering
    \begin{tabular}{c|c|c|c|c|c}
$\ell$ & $|C_i|$ & $v$ & $k$ & $\lambda_{\min}$ & Existence? \\ \hline
$4$ & $4$ & $16$ & $6$ & $2$ & Theorems~\ref{thm: Hadamard matrix gives 0-ULSE coloured BIBD} and \ref{thm: affine plane gives 0-LSE l-col BIBD}\\
$5$ & $3$ & $15$ & $8$ & $4$ & Theorem~\ref{thm: twin prime powers give 0-ULSE BIBD}\\
$5$ & $9$ & $45$ & $12$ & $3$ & Theorem~\ref{thm: affine plane gives 0-LSE l-col BIBD}\\
$6$ & $6$ & $36$ & $15$ & $6$ & \\
$6$ & $16$ & $96$ & $20$ & $4$ & Theorem~\ref{thm: affine plane gives 0-LSE l-col BIBD}\\
$7$ & $5$ & $35$ & $18$ & $9$ & Theorem~\ref{thm: twin prime powers give 0-ULSE BIBD}\\
$7$ & $10$ & $70$ & $24$ & $8$ & \\
$7$ & $25$ & $175$ & $30$ & $5$ & Theorem~\ref{thm: affine plane gives 0-LSE l-col BIBD}\\
$8$ & $8$ & $64$ & $28$ & $12$ & Theorem~\ref{thm: Hadamard matrix gives 0-ULSE coloured BIBD}\\
$8$ & $15$ & $120$ & $35$ & $10$ & \\
$8$ & $36$ & $288$ & $42$ & $6$ & \\
$9$ & $7$ & $63$ & $32$ & $16$ & Theorem~\ref{thm: twin prime powers give 0-ULSE BIBD}\\
$9$ & $21$ & $189$ & $48$ & $12$ & \\
$9$ & $49$ & $441$ & $56$ & $7$ & Theorem~\ref{thm: affine plane gives 0-LSE l-col BIBD}\\
$10$ & $4$ & $40$ & $27$ & $18$ & Theorem~\ref{thm: there exists a 0-ULSE (40,27,18)-BIBD}\\
$10$ & $10$ & $100$ & $45$ & $20$ & \\
$10$ & $16$ & $160$ & $54$ & $18$ & \\
$10$ & $28$ & $280$ & $63$ & $14$ & \\
$10$ & $64$ & $640$ & $72$ & $8$ & Theorem~\ref{thm: affine plane gives 0-LSE l-col BIBD}\\
$11$ & $6$ & $66$ & $40$ & $24$ & \\
$11$ & $9$ & $99$ & $50$ & $25$ & Theorem~\ref{thm: twin prime powers give 0-ULSE BIBD}\\
$11$ & $21$ & $231$ & $70$ & $21$ & \\
$11$ & $36$ & $396$ & $80$ & $16$ & \\
$11$ & $81$ & $891$ & $90$ & $9$ & Theorem~\ref{thm: affine plane gives 0-LSE l-col BIBD}\\
$12$ & $12$ & $144$ & $66$ & $30$ & Theorem~\ref{thm: Hadamard matrix gives 0-ULSE coloured BIBD}\\
$12$ & $45$ & $540$ & $99$ & $18$ & \\
$12$ & $100$ & $1200$ & $110$ & $10$ & \\
$13$ & $11$ & $143$ & $72$ & $36$ & Theorem~\ref{thm: twin prime powers give 0-ULSE BIBD}\\
$13$ & $22$ & $286$ & $96$ & $32$ & \\
$13$ & $33$ & $429$ & $108$ & $27$ & \\
$13$ & $55$ & $715$ & $120$ & $20$ & \\
$13$ & $121$ & $1573$ & $132$ & $11$ & Theorem~\ref{thm: affine plane gives 0-LSE l-col BIBD}\\
$14$ & $14$ & $196$ & $91$ & $42$ & \\
$14$ & $27$ & $378$ & $117$ & $36$ & \\
$14$ & $40$ & $560$ & $130$ & $30$ & \\
$14$ & $66$ & $924$ & $143$ & $22$ & \\
$14$ & $144$ & $2016$ & $156$ & $12$ &
    \end{tabular}
    \caption{Admissible parameters for symmetric $0$-ULSE $\ell$-coloured $(v,k,\lambda_{\min})$-BIBDs where $4\leq \ell \leq 14$.}
    \label{tab: k=0 mod l-1, possible BIBDs}
\end{table}

\begin{question}
    How can we find more families of $0$-ULSE coloured BIBDs?
\end{question}

One possible candidate for a new family comes from projective geometries. Consider an odd integer $d\geq 3$. The complement of a projective geometry PG$_d(q)$ is a \[\left(\frac{q^{d+1} - 1}{q-1}, q^d, q^{d-1}(q-1)\right)\text{-BIBD}.\] It is known that there exists a projective geometry PG$_d(q)$ for all $d\geq 2$ and all prime powers $q$; see Section~2.3 of \cite{S04} for a proof. It can be verified that this BIBD satisfies Equation~\eqref{eq: k=0 mod l-1, soln for v} when $\ell = \sqrt{q^{d+1}} + 1$. It is impossible to create the ingredient designs in Theorem~\ref{thm: general construction for nice 0-LSE l-coloured BIBDs} from a $\left(\frac{q^{d+1} - 1}{q-1}, q^d, q^{d-1}(q-1)\right)$-BIBD. The necessary RBIBD ingredient fails to exist since the block size does not divide the number of points, as is required of an RBIBD. So if $0$-ULSE coloured $\left(\frac{q^{d+1} - 1}{q-1}, q^d, q^{d-1}(q-1)\right)$-BIBDs exist, they would constitute another infinite family that is not captured by Theorem~\ref{thm: general construction for nice 0-LSE l-coloured BIBDs}. In support of the hypothesis that this family does exist, we have the following result.

\begin{theorem}\label{thm: there exists a 0-ULSE (40,27,18)-BIBD}
    There exists a $0$-ULSE $\left(\sqrt{q^{d+1}} + 1\right)$-coloured $\left(\frac{q^{d+1} - 1}{q-1}, q^d, q^{d-1}(q-1)\right)$-BIBD for $q = d = 3$.
\end{theorem}

\begin{proof}
    Begin with the BIBD obtained from the $(40, 13, 4)$ difference set \[\{0,1,2,4,5,8,13,14,17,19,24,26,34\}\] in $\mathbb{Z}_{40}$. Next, partition the points in $\mathbb{Z}_{40}$ by the colour classes $C_i = \{i, 10 + i, 20 + i, 30 + i\}$ for each $0\leq i\leq 9$. This construction forms a $4$-ULSE $10$-coloured $(40,13,4)$-BIBD. Taking the complement of this $4$-ULSE $10$-coloured $(40,13,4)$-BIBD produces a $0$-ULSE $10$-coloured $(40,27,18)$-BIBD. 
\end{proof}

\begin{question}
    Does there exist a $0$-ULSE $(\sqrt{q^{d+1}} + 1)$-coloured $\left(\frac{q^{d+1} - 1}{q-1}, q^d, q^{d-1}(q-1)\right)$-BIBD for all prime powers $q$ and for all odd $d\geq 3$?
\end{question}

Another possible candidate is a family of BIBDs with parameters \[\left(\frac{1}{2}\ell(\ell-2)(\ell-3), (\ell-1)(\ell-3), 2(\ell-3)\right)\] where $\ell \geq 5$. It can be verified that such a BIBD satisfies Equation~\eqref{eq: k=0 mod l-1, soln for v}. However, such a family cannot be obtained by Theorem~\ref{thm: general construction for nice 0-LSE l-coloured BIBDs} since it would require a transversal design of index $\frac{2(\ell-3)}{\ell-2}$, which is not an integer if $\ell \geq 5$.

\begin{question}
    Does there exist a $0$-ULSE $\ell$-coloured $\left(\frac{1}{2}\ell(\ell-2)(\ell-3), (\ell-1)(\ell-3), 2(\ell-3)\right)$-BIBD for all $\ell\geq 5$?
\end{question}

Not every BIBD satisfying Equation~\eqref{eq: k=0 mod l-1, soln for v} has a $0$-ULSE colouring. As mentioned in Section~\ref{sec: Necessary Conditions}, there are five $(15,8,4)$-BIBDs up to isomorphism and by computer search we have determined that only one of them has a $0$-ULSE $5$-colouring. However, we have yet to find integers $v,k,\lambda, \ell \in \mathbb{Z}^+$ that satisfy Equation~\eqref{eq: k=0 mod l-1, soln for v} and Lemmas \ref{lem: replication num of a BIBD} and \ref{lem: num of blocks in a BIBD} such that there is no $0$-ULSE $\ell$-coloured symmetric $(v,k,\lambda)$-BIBD. This motivates the following conjecture.

\begin{conjecture}\label{conj: existence of 0-ULSE sym BIBDs}
    For every $v,k,\ell \in \mathbb{Z}^+$ satisfying Equation~\eqref{eq: k=0 mod l-1, soln for v} such that there exists a symmetric BIBD on $v$ points with block size $k$, there exists a $0$-ULSE $\ell$-coloured symmetric BIBD on $v$ points with block size $k$.
\end{conjecture}

The Bruck-Ryser-Chowla Theorem, proven by the work done in \cite{BR49, CR50, S49}, states that if a symmetric $(v,k,\lambda)$-BIBD exists, then $k-\lambda$ is a perfect square if $v$ is even and there exists a nontrivial solution to \[x^2 = (k-\lambda)y^2 + (-1)^{\frac{v-1}{2}}\lambda z^2\] if $v$ is odd. From Theorem~\ref{thm: symmetric 0-ULSE lambda value}, in a symmetric $0$-ULSE $\ell$-coloured $(v,k,\lambda)$-BIBD, $\lambda = \frac{k((\ell-1)^2 - k)}{(\ell-1)^2}$. Therefore $k-\lambda = \left(\frac{k}{\ell-1}\right)^2$. Furthermore, $(x,y,z) = (k, \ell-1, 0)$ is a nonzero solution to the above equation. So, the Bruck-Ryser-Chowla Theorem does not prove the nonexistence of any symmetric $0$-ULSE coloured BIBD.

Our work gives new necessary conditions for the existence of Hadamard matrices, affine planes, and twin prime powers. As mentioned in Section~\ref{sec: Applications of general construction}, determining when Hadamard matrices exist and determining when affine planes exist are two long-standing open problems. Whether there exist infinitely many twin prime powers, i.e. solutions to the equation
\begin{equation}\label{eq: TPP equation}
    p^n - q^m = 2
\end{equation}
where $p$ and $q$ are primes and $n,m\in \mathbb{Z}^+$, is also an open problem. If $n=m=1$ in \eqref{eq: TPP equation}, then $p$ and $q$ are twin primes. For background on the Twin Prime Conjecture, which asserts that there are infinitely many twin primes, see \cite{M19}. Solutions for $p$ when $n = 1$ and $m\geq 1$ are listed in sequence A267945 of the On-Line Encyclopedia of Integer Sequences (OEIS) \cite{OEISc} and solutions for $q$ when $n \geq 1$ and $m=1$ are listed in sequence A267944 in the OEIS \cite{OEISb}. If $n>1$ and $m>1$, then \eqref{eq: TPP equation} is a special case of Pillai's Conjecture \cite{P36, P37}, which states that for any fixed $c \in \mathbb{Z}^+$, there are finitely many solutions to the equation $x^n - y^m = c$ where $x,y,n,m\in \mathbb{Z}^+$ and $n,m \geq 2$. There is only one known solution to $x^n - y^m = 2$ (see \cite{OEISa}) which is $(x,y,n,m) = (3,5,3,2)$; the solution $(p,q,n,m) = (3,5,3,2)$ also satisfies \eqref{eq: TPP equation}. For the number of known solutions for each $c\in \mathbb{Z}^+$, see the sequence A076427 in the OEIS \cite{OEISa}.

It is well known that there are infinitely many Hadamard matrices and affine planes (see \cite{S04}) and so the families of $0$-ULSE coloured BIBDs in Theorems~\ref{thm: Hadamard matrix gives 0-ULSE coloured BIBD} and \ref{thm: affine plane gives 0-LSE l-col BIBD} are infinite families. Since we are interested in determining if the family of $0$-ULSE coloured BIBDs in Theorem~\ref{thm: twin prime powers give 0-ULSE BIBD} is also infinite, we ask the following question. 

\begin{question}
    Are there infinitely many twin prime powers?
\end{question}

By Theorems~\ref{thm: Hadamard matrix gives 0-ULSE coloured BIBD}, \ref{thm: affine plane gives 0-LSE l-col BIBD}, and \ref{thm: twin prime powers give 0-ULSE BIBD}, if a Hadamard matrix, an affine plane, or a pair of twin prime powers is given, then a $0$-ULSE coloured BIBD can be constructed. Consequently, the existences of the families of $0$-ULSE coloured BIBDs stated in Theorems~\ref{thm: Hadamard matrix gives 0-ULSE coloured BIBD} and \ref{thm: affine plane gives 0-LSE l-col BIBD} are necessary conditions for the existences of Hadamard matrices and affine planes respectively. Hence, proving the nonexistence of these $0$-ULSE coloured BIBDs of specific orders would prove the nonexistence of the corresponding Hadamard matrices and affine planes. Furthermore, proving that there are only finitely many of the $0$-ULSE coloured BIBDs from Theorem~\ref{thm: twin prime powers give 0-ULSE BIBD} would prove that there are only finitely many twin prime powers.

\begin{question}
    Can $0$-ULSE coloured BIBDs be used to show the nonexistence of Hadamard matrices, affine planes, and twin prime powers?
\end{question}

\section*{Acknowledgements}

Burgess is supported by NSERC (grant number
RGPIN-2025-04633), Kellough is supported by an NSERC Postgraduate Scholarship – Doctoral, and Pike is supported by NSERC (grant number RGPIN-2022-03829).

\end{document}